\documentclass[leqno,11pt]{amsart}

\usepackage{hyperref}
\usepackage{enumitem}
\usepackage[active]{srcltx}
\usepackage{amsmath}
\usepackage{amssymb}
\usepackage{amscd}
\usepackage{amsthm}
\usepackage{cleveref}
\usepackage{pdfsync}
\usepackage{tikz}
\usepackage{tikz-cd}
\usepackage{mathrsfs}
\usepackage[latin1]{inputenc}
\usepackage{graphics}

\newtheorem{teo}[equation]{Theorem}
\newtheorem{defin}[equation]{Definition}
\newtheorem{prop}[equation]{Proposition}
\newtheorem{cor}[equation]{Corollary}
\newtheorem{lemma}[equation]{Lemma}
\newtheorem{prob}[equation]{Problem}
\theoremstyle{definition}

\newtheorem{ex}[equation]{Example}
\newtheoremstyle{dico}% name of the style to be used
 {\baselineskip}   % ABOVESPACE
  {\topsep}   % BELOWSPACE
  {}  % BODYFONT
  {0pt}       % INDENT (empty value is the same as 0pt)
  {} % HEADFONT
  {.}         % HEADPUNCT
  {5pt plus 1pt minus 1pt} % HEADSPACE
  {}          % CUSTOM-HEAD-SPEC
\theoremstyle{dico}
\newtheorem{say}[equation]{}
\numberwithin{equation}{section}

\newcommand{\XS}{\overline{X}^S_\tau}
\newcommand{\e}{\operatorname{e}}
\newcommand{\Keler}             {K\"{a}hler }

\newcommand{\inte}{\operatorname{int}}
\newcommand{\dalf}{\dot{\alfa}}
\newcommand{\ga}{\gamma}
\newcommand{\dga}{\dot{\gamma}}

\newcommand{\OO}{\mathcal{O}}
\newcommand{\meno}{^{-1}}

\newcommand{\PP}{\mathbb{P}}
\newcommand{\Mom}{\mu}

\newcommand{\proba}{\mathscr{P}}
\newcommand{\pb}{\proba(M)}
\newcommand{\misu}{\mathscr{M}}

\newcommand{\mis}{{\nu}}
\newcommand{\kn}{\psim}

\newcommand{\tnu}{\tilde{\nu}}
\newcommand{\tf}{\widehat{F}}

\newcommand{\convo}{E( \mu)}

\newcommand{\intec}{\Omega(\mu)}
\newcommand{\into} {\intec}
\newcommand{\XX}{\mathfrak{X}}

\newcommand{\psiM}{\psi^\proba}
\newcommand{\PsiM}{\Psi^\proba}
\newcommand{\Psim}{\Psi^M}
\newcommand{\psim}{\psi^M}

\newcommand{\liu}{\mathfrak{u}}
\newcommand{\liek}{\mathfrak{k}}
\newcommand{\liel}{\mathfrak{l}}
\newcommand{\lieg}{\mathfrak{g}}
\newcommand{\liet}{\mathfrak{t}}
\newcommand{\liep}{\mathfrak{p}}
\newcommand{\lieq}{\mathfrak{q}}
\newcommand{\lies}{\mathfrak{s}}
\newcommand{\liez}{\mathfrak{z}}
\newcommand{\lia}{\mathfrak{a}}

\newcommand{\grad}{\operatorname{grad}}
\newcommand{\chern}{\operatorname{c}}

\newcommand{\alfa}{\alpha}
\newcommand{\alf}{\alpha}
\newcommand{\vacuo}{\emptyset}

 \newcommand{\pf}{{}_*}

\newcommand{\la}{\lambda}

\newcommand{\enf}{\emph}

\newcommand{\desudt}[1] []      {\dfrac {\mathrm {d} #1 }{\mathrm {dt}}}
\newcommand{\desudtzero}        {\desudt \bigg \vert _{t=0} }
\newcommand{\deze}        {\desudt \bigg \vert _{t=0} }

\newcommand{\restr}[1]          {\vert_{#1}}

\newcommand{\Vol}{\operatorname{Vol}}
\newcommand{\vol}{\operatorname{vol}}

\newcommand{\Ad}{\operatorname{Ad}}

\newcommand{\sx}{\langle}
\newcommand{\xs}{\rangle}
\newcommand{\scalo}{\sx \cdot , \cdot \xs}

\newcommand{\Aut}{\operatorname{Aut}}

\newcommand{\SU} {\operatorname{SU}}

\newcommand{\su} {\mathfrak{su}}
\newcommand{\Sl}{\operatorname{SL}}
\newcommand{\SL}{\operatorname{SL}}
\newcommand{\Gl}{\operatorname{GL}}

\newcommand{\lds}{\ldots}
\newcommand{\cds}{\cdots}
\newcommand{\cd}{\cdot}
\renewcommand{\setminus}{-}
\newcommand{\cinf}{C^\infty}

\newcommand{\ra}{\rightarrow}
\newcommand{\lra}{\longrightarrow}
\newcommand{\C}{\mathbb{C}}
\newcommand{\R}{\mathbb{R}}

\newcommand{\om}{\omega}

\newcommand{\eps}{\varepsilon}
\renewcommand{\phi}{\varphi}
\renewcommand{\bigl}{\left}
\renewcommand{\biggl}{\left}
\renewcommand{\Bigl}{\left}
\renewcommand{\bigr}{\right}
\renewcommand{\biggr}{\right}
\renewcommand{\Bigr}{\right}

\newcommand{\fun}{\mathfrak{F}}
\newcommand{\bly}{F}

\newcommand{\Crit}{\operatorname{Crit}}
\newcommand{\Weyl}{\mathcal{W}}
\newcommand{\Id}{\operatorname{id}}
\newcommand{\id}{\operatorname{id}}
\newcommand{\lieh}{\mathfrak{h}}

\newcommand{\spaz}{\mathscr{M}}
\newcommand{\mume}{\fun\meno(0)}
\newcommand{\unfi}{u_\infty}
\newcommand{\tits}{\partial_\infty X}
\newcommand{\limes}{\alfa}

\newcommand{\x}{{v}}

\newcommand{\W}{\mathscr{W}}
\newcommand{\ac}{\mathscr{AC}}
\newcommand{\weak}{\rightharpoonup}
\newcommand{\go}{\hat{g}}
\newcommand{\fo}{\hat{f}}

\begin{document}

\title{Stability of measures on K\"ahler manifolds}

\author{Leonardo Biliotti}

\author{Alessandro Ghigi}

\begin{abstract}
  Let $(M,\om)$ be a K\"ahler manifold and let $K$ be a compact group
  that acts on $M$ in a Hamiltonian fashion.  We study the action
  of $K^\C$ on probability measures on $M$.
  First of all we identify an abstract setting for the momentum
  mapping and give numerical criteria for stability, semi-stability
  and polystability. Next we apply this setting to the action of
  $K^\C$ on measures.  We get various stability criteria for measures
  on K\"ahler manifolds. The same circle of ideas gives a very general
  surjectivity result for a map originally studied by Hersch and
  Bourguignon-Li-Yau.
\end{abstract}

\address{Universit\`{a} degli Studi di Parma} \email{leonardo.biliotti@unipr.it}
\address{Universit\`a degli Studi di Pavia}\email{alessandro.ghigi@unipv.it}

\thanks{The authors were partially supported by FIRB 2012 ``Geometria
  differenziale e teoria geometrica delle funzioni'', by a grant of
  the Max-Planck Institut f\"ur Mathematik, Bonn and by GNSAGA of
  INdAM.
 The first author was also supported by MIUR PRIN  2010-2011  
   ``Variet\`a reali e complesse: geometria, topologia e analisi armonica''.
 The second author was also supported by MIUR PRIN   2012
  ``Moduli, strutture geometriche e loro applicazioni''. }



\maketitle

\tableofcontents

\section{Introduction}
\label{sec:introduction}

Let $(M,\om)$ be a \Keler manifold and let $K$ a compact connected Lie
group. Assume that $K$ acts on $M$ in a Hamiltonian way with momentum
mapping $\mu : M \lra \liek^*$.  For $\nu$ a positive measure on $M$
set
\begin{gather}
  \label{bly-intro}
  \fun(\nu) : = \int_M \mu(x) d\nu(x).
\end{gather}
This defines a map from the set of measures to $\liek^*$. This map has
been studied at different levels of generality by Hersch
\cite{hersch}, Millson and Zombro \cite{millson-zombro},
Bourguignon, Peter Li and Yau \cite{bourguignon-li-yau} and ourselves
\cite{biliotti-ghigi-American}.  Recall that $G:=K^\C$ acts on $M$ and
hence on the set of measures on $M$.  As in the quoted papers, we are
interested in the following problem:
\begin{prob}
  \label{problem}
  \nonumber Assume that $0$ belongs to the interior of the convex
  envelope of $\mu(M) $ and let $\nu $ be a measure on $M$. Is there
  $g\in G$ such that $\fun(g\cd \nu) =0$?
\end{prob}
This question is motivated by an application to upper bounds for the
first eigenvalue of the Laplacian on functions.  For more details see
\ref {eigen} or the introduction to \cite{biliotti-ghigi-American}.
In this paper we concentrate on Problem \ref{problem} leaving aside
the applications to eigenvalue estimates.

For a sufficiently regular measure a positive answer to Problem 1 is known in a few
cases, namely $M=\PP^1$ (Hersch in 1970 \cite{hersch}), $M=\PP^n$
(Bourguignon, Li and Yau in 1994 \cite{bourguignon-li-yau}) and $M$ a
flag manifold  (ourselves in 2013
\cite{biliotti-ghigi-American}).  In all these cases $\om$ is the
symmetric metric and $K$ is the connected component of the isometry
group.

Our main theorem is a rather vast generalization of these results.
\begin{teo}
  \label{poppa}
  The answer to Problem \ref{problem} is positive for an arbitrary
  \Keler manifold with a Hamiltonian action of $K$ and for any
  probability measure $\nu$ that is absolutely continuous with respect
  to smooth strictly positive measures.
\end{teo}
(See \ref{smooth-measure} and Definition \ref{def-ac} for the
definition of this class of measures. See Theorem \ref{propria} for
the actual result which is in fact stronger than stated here).

To study Problem \ref{problem} we cast it in a momentum mapping
picture.  The natural action to consider is the action of $G$ on the
set $\proba(M)$ of Borel probability measures on $M$. This set does
not seem to admit a reasonable symplectic structure. (But see
\cite{gangbo-kim-pacini} for something similar in the case of
Euclidean space.) One could find some space admitting a symplectic
structure and fibering over $\proba(M)$ and one could lift the problem
to this space. But this seems rather artificial. Instead it turns out
that the part of the momentum mapping/G.I.T.  picture which is needed
can be developed on a topological space with no symplectic structure.
Thus our first goal is to build up a theory for the momentum mapping
that can be applied to the action of $G$ on $\pb$.  This is the
content of Sections
  \ref{sec:abstract-setting}-\ref{sec:polyst-semi-stab}.

More precisely, 
given a Hausdorff topological space with a continuous $G$-action and a
set of functions formally similar to the classical Kempf-Ness
functions (see \ref{setting}) we define an analogue of the momentum
mapping and the usual concepts of stability (see Definition \ref
{stabilita}). The point of this construction is that one can characterize
stability, semi-stability and polystability of a point $x$ by numerical criteria,
that is in terms of a function called \emph{maximal weight} and
  denoted $\la_x$, which is defined on Tits boundary of $G/K$. See
\ref{say-def-la} for the definition of $\la_x$ and Theorems
\ref{stabile}, \ref{semi-stable-abstract}, \ref{ps} for the precise
statements. We also establish a version of the Hilbert-Mumford criterion (Corollary \ref{Hilbert-Mumford}) and the openness of the set of stable points (Corollary \ref{stable-open-abstract-setting}).
In the classical case of a group action
on a K\"ahler manifold these characterizations are due to Mundet i Riera \cite{mundet-Crelles, mundet-Trans}, Teleman \cite {teleman-symplectic-stability},
Kapovich, Leeb and Millson \cite{kapovich-leeb-millson-convex-JDG} and
probably many others. In fact many of these ideas go back as far as
Mumford \cite[\S 2.2]{mumford-GIT}.

A different approach to stability on K\"ahler manifolds is based on
the properties of invariant plurisubharmonic functions and the Mostow
fibration instead of Kempf-Ness functions. This has been developed
over the years by Azad, Heinzner, Huckleberry, Loeb, Loose, Schwarz
and others \cite{azad-loeb-bulletin}, \cite{heinzner-GIT-stein}, \cite
{heinzner-huckleberry-Inventiones}, \cite{heinzner-huckleberry-MSRI},
\cite {heinzner-huckleberry-loose}, \cite{heinzner-loose},
\cite{heinzner-schwarz-Cartan}.  It seems hard to apply this approach
in our setting since it relies heavily on the fact that the space
where the group acts is a complex space.  A similar remark applies to
the techniques used in \cite{georgula}.  The main tool there is the
gradient flow of the momentum mapping squared. This is not available
without some kind of differentiable structure.

The setting we have chosen to develop the theory is not necessarily
the most natural, nor the most general. However it is well suited for
the study of our problem, namely the stability of probability measures
on a compact K\"ahler manifold.  It should be of interest in other
situations. Even in the classical case of a Hamiltonian action on a
compact K\"ahler manifold, it offers a rather streamlined proof of the
numerical criteria for stability, semi-stability and polystability.

In Section \ref{sez-misure} we apply the abstract theory to the action
of $G$ on $\pb$ endowed with the weak topology.  The map
\eqref{bly-intro} turns out be the analogue of the momentum mapping in
this setting.  Using the Morse-Bott theory of the momentum mapping on
$M$ we are able to compute rather explicitly the maximal weight
$\la_\nu$ of $\nu \in \pb$.  Indeed fix a non-zero $v\in \liek$ and
let $c_0 < \cds < c_r$ be the critical values of
$\mu^v := \sx \mu, v\xs$. Let $C_i:=(\mu^v)\meno ( c_i)$ be the
critical components.  Set
\begin{gather*}
  W^u_i := \{ x\in M: \lim_{t\to +\infty} \exp(itv) \cd x \in C_i \}.
\end{gather*}
This is the \emph{unstable manifold} of $C_i$ for the gradient flow of
$\mu^v$.  Denote by $\e(-v)$ the point of $\tits$ corresponding to the
geodesic $t \mapsto \exp(-itv)\cd K$ in $G/K$.  The maximal weight can
be computed in terms of these Morse data and the result is rather
clean:
\begin{gather*}
  \la_\nu(\e(-v)) = \sum_{i=0}^r c_i \cd \nu(W^u_i).
\end{gather*}
Based on this formula we get various stability criteria for measures.
For example every measure that is absolutely continuous with respect
to a smooth strictly positive measure is stable up to shifting the
momentum mapping (Theorem \ref{W-stabile}).

In Section \ref{sec:bly} we go back to Problem \ref{problem}.
The first step is analogous to something known in the
finite-dimensional case: if the stabilizer of a measure $\nu$ is
compact, then the restriction of $\fun$ to the orbit $G\cd \nu$ is a
submersion (Theorem \ref{maxirank}).  Under a mild regularity
assumption on $\nu$ the restriction of $\fun$ to $G\cd \nu$ is in fact
a smooth fibration onto the interior of the convex envelope of
$\mu(M)$ (Theorem \ref{propria}).  Theorem \ref{poppa} follows
immediately.

Section \ref{sec:applications} contains some applications.  Using the
previous results we get a precise characterization of stable,
semi-stable and polystable measures on $\PP^n$ (Theorems \ref{prost}
and \ref{props}), extending previous work by Millson-Zombro and
Donaldson.

Next we turn to the application to upper bounds for $\la_1$. We
explain that the results in the paper give shorter and more conceptual
proofs of some known statements. We also explain what is missing to
get a very general estimate for $\la_1$ using these ideas.

{\bfseries \noindent{Acknowledgements.}}  The authors wish to thank
Luigi Fontana, Daniel Greb, Peter Heinzner and Ignasi Mundet i Riera
for interesting discussions/emails related to the subject of this
paper.  They also wish to thank the Max-Planck Institut f\"ur
Mathematik, Bonn for excellent conditions provided during their visit
at this institution, where they started working on the subject of this
paper. Finally they wish to thank the referees for reading very
carefully the manuscript.

\section{Kempf-Ness functions}
\label{sec:abstract-setting}
In this section we introduce an abstract setting for actions of
complex reductive groups on topological spaces. More precisely, we
define functions similar to the classical Kempf-Ness functions, from
which we derive the whole momentum mapping picture. We start with some
remarks on convex functions of symmetric spaces.

\begin{say}
  \label{ss1}
  Let $X$ be a symmetric space of the noncompact type. (Some of the
  following statements hold more generally if $X$ is a Hadamard
  manifold or even a $ CAT(0)$--space. We restrict to the case of a
  symmetric space which is the only one needed for the theory of the
  momentum mapping.)  Two unit speed geodesics $\ga, \ga' : \R \ra X$
  are equivalent, denoted $\ga \sim \ga'$, if
  $\sup_{t>0} d(\ga(t), \ga'(t)) $ $ < +\infty$.  The \emph{Tits
    boundary} of $X$, denoted by $\tits$, is the set of equivalence
  classes of unit speed geodesics in $X$.  Assume that $X=G/K$ with
  $G$ is a connected Lie group and $K\subset G$ a maximal compact
  subgroup. Put $o:= K \in X$.  For $v \in \liek $, let $\ga^v$ denote
  the geodesic $\ga^v (t): = \exp(it\x) K$. Mapping $v$ to the tangent
  vector $\dot{\ga}^v(0)$ yields an isomorphism $\liek \cong T_oX$.
  Since any geodesic ray in $X$ is equivalent to a unique ray starting
  from $o$, the map
  \begin{gather}
    \label{def-e}
    \e : S(\liek) \ra \tits, \ \e(v):= [\ga^v],
  \end{gather}
  where $S(\liek)$ is the unit sphere in $\liek$, is a bijection. The
  \emph{sphere topology} is the topology on $\tits$ such that $\e$ is
  a homeomorphism.  (For more details on the Tits boundary see for
  example \cite[\S I.2] {borel-ji-libro}.)
\end{say}

\begin{say}
  \label{ss2}
  Let $X$ be a topological space. A continuous function $u: X \ra \R$
  is an \emph{exhaustion} if for any $c\in \R$ the set
  $f\meno ((-\infty, c])$ is compact. A continuous function is an
  exhaustion if and only if it is bounded below and proper.  The
  following lemma is proven in greater generality in the paper
  \cite[\S 3.1]{kapovich-leeb-millson-convex-JDG} by Kapovich, Leeb
  and Millson. Since it is basic to everything that follows, we recall
  its proof in detail. For more information on the geometry of convex
  functions on symmetric spaces and the relation with their Tits
  boundary see \cite{kapovich-leeb-millson-convex-JDG} especially
  pp. 313-316.
\end{say}

\begin{lemma}\label{convex-function}
  Let $u: X \ra \R$ be a smooth convex function on $X$. If
  $u$ is globally Lipschitz,  the function % $\unfi : \tits \ra \R$
  \begin{gather}
    \label{eq:1}
    \unfi : \tits \ra \R ,\quad \unfi ([\ga]) : = \lim _{t \to +\infty
    } (u\circ\ga)'(t),
  \end{gather}
  is well--defined. Moreover $u$ is an exhaustion if and only if
  $\unfi > 0$ on $\tits$.
\end{lemma}

\begin{proof}
  Since $f:=u\circ \ga$ is convex,
  \begin{gather*}
    \frac{f(s)}{s} \leq f'(s) \leq \frac{f(t) - f(s)}{t-s} \quad
    \text{ for }0< s< t.
  \end{gather*}
  Moreover the first two quantities are increasing in $s$, the third
  in $t$.  Thus
  \begin{gather*}
    \lim_{s\to +\infty } \frac{f(s)}{s} \leq \lim_{s \to +\infty}
    f'(s) \leq \lim_{t\to +\infty} \frac{f(t) - f(s) } { t -s } .
  \end{gather*}
  Since the last limit equals the first we get
  $\lim _{t \to +\infty } f'(t) = \lim_{t\to +\infty} f(t) / t$.  If
  $\tilde\ga$ is another geodesic and
  $d(\ga(t), \tilde \ga(t)) \leq C$, then
  $| u\circ\ga(t) / t - u\circ \tilde \ga(t) / t| \leq LC / t$, where
  $L$ is a Lipschitz constant for $u$. Therefore
  \begin{gather*}
    \lim _{t \to +\infty } (u\circ \ga)'(t) = \lim _{t \to +\infty }
    \frac{u\circ\ga(t)}{t}= \lim _{t \to +\infty } \frac{u\circ \tilde
      \ga(t)}{t}= \lim _{t \to +\infty } (u\circ \tilde \ga)'(t).
  \end{gather*}
  This shows that $\unfi$ is well-defined.  It is finite since
  $|f'(s) | \leq L$.  Assume now that $u$ is an exhaustion.  Given
  $\ga$ there is $t_0>0$ such that for $t> t_0$,
  $u\circ\ga(t) > u\circ\ga(0)$. Thus
  \begin{gather*}
    \unfi[ \ga] = \lim_{t\to +\infty} \frac{ u\circ \ga(t) -
      u\circ\ga(0) } {t} >0.
  \end{gather*}
  Conversely assume that $\unfi ([\ga]) > 0$ for any $[\ga]\in \tits$.
  Fix $x\in X$ and let $S$ be the unit sphere in $T_xX$. For any
  $v\in S$ let $\ga^v(t) =\exp_x(tv)$. Since $\unfi([\ga^v]) >0$ there
  are $\eps (v) >0$ and $T_v $ such that
  $u\circ\ga^v(t) \geq 2\eps(v) t $ for any $t\geq T_v$.  By
  continuity there is a neighbourhood $U_v \subset S$ of $v$ such that
  $u(\exp_x (T_vw)) \geq \eps(v) T_v $ for any $w\in U_v$. Since
  $u\circ \ga_w$ is convex, the function $ u\circ\ga_w(t) / t$ is
  increasing. Thus $u(\exp_x (tw) ) \geq \eps(v) t $ for any
  $t\geq T_v$ and any $w\in U_v$.  If $S = \bigcup_{i=1}^k U_{v_i}$
  set $T:=\max \{T_{v_i}\}$ and $\eps : = \min \{\eps(v_i)\}$. Then
  $u ( \exp_x v) \geq \eps |v| $ for any $v\in T_xX$ with
  $|v| \geq T$. This shows that $u$ is an exhaustion function.
\end{proof}

\begin{say}
  \label{setting}
  Let $\spaz$ be a Hausdorff topological space and let $K$ be a
  compact connected Lie group. Denote by $G=K^\C$ the complexification
  of $K$ and assume that $G$ acts from the left on $\spaz$ and that
  the action is continuous, that is the map $G \times \spaz \ra \spaz$
  is continuous.  Starting with these data we are going to consider a
  function $ \Psi : \spaz \times G \ra \R$, subject to six conditions.
  The first four conditions are the following ones:
  \begin{enumerate}[label=(P\arabic*),ref=(P\arabic*)]
  \item \label{P1} For any $x\in \spaz$ the function $ \Psi(x,\cd )$
    is smooth on $G$.
  \item \label{P2} The function $\Psi(x, \cd )$ is left--invariant
    with respect to $K$: $\Psi(x,kg) = \Psi(x,g)$.
  \item\label{P3} For any $x\in \spaz$, and any $\x \in \liek$ and
    $t\in \R$
    \begin{gather*}
      \frac{\mathrm{d^2}}{\mathrm{dt}^2 } \Psi(x,\exp(it\x)) \geq 0.
    \end{gather*}
    Moreover
    \begin{gather*}
      \frac{\mathrm{d^2}}{\mathrm{dt}^2 }\bigg \vert_{t=0}
      \Psi(x,\exp(it\x)) = 0
    \end{gather*}
    if and only if $\exp(\C \x) \subset G_x$.
  \item \label{P4} For any $x\in \spaz$, and any $g, h\in G$
    \begin{gather*}
      \Psi(x,g) + \Psi({gx}, h) = \Psi(x,hg).
    \end{gather*}
    (This equation is called the \emph{cocycle condition}.)
  \end{enumerate}
\end{say}
\begin{say}
  Set
  \begin{gather*}
    X:=G/K.
  \end{gather*}
  Fix an $\Ad$--invariant scalar product $\scalo$ on $\liek$ and
  consider the corresponding Riemannian metric on $X$.  If $\Psi$ is a
  function satisfying \ref{P1}--\ref{P4}, then by \ref{P2} the
  function $g\mapsto \Psi(x, g\meno)$ descends to a function on $X$:
  \begin{gather}
    \label{defpsi}
    \psi_x: X \ra\R, \quad \psi_x(gK) := \Psi(x, g\meno).
  \end{gather}
  Using $\psi_x$ instead of $\Psi$ the cocycle condition reads
  \begin{gather}
    \tag{P4$'$}
    \label{cociclo-psi}
    \psi_x(ghK) = \psi_x (gK) + \psi_{g\meno x} (hK).
  \end{gather}
  We can now state our fifth assumption:
  \begin{enumerate}[label=(P\arabic*),ref=(P\arabic*)]
    \setcounter{enumi}{4}
  \item \label{P5} For all $x \in \spaz$, the function $\psi_x$ is
    globally Lipschitz on $X$.
  \end{enumerate}
\end{say}

\begin{say}
  Let $\scalo : \liek^*\times \liek \ra \R$ be the duality pairing.
  For $x\in \spaz$ define $\fun(x) \in \liek^*$ by requiring that
  \begin{multline}
    \label{momento-astratto}
    \sx \fun (x), \x\xs = -  d\psi_x (o)(\dga^\x(0)) = \\
    = \desudtzero \psi_x( \exp(-it\x) K) = \desudtzero \Psi(x,
    \exp(it\x)) .
  \end{multline}
  (Notation as in \ref{ss1}.)  The following is the last condition
  imposed on the function $\Psi$:
  \begin{enumerate}[label=(P\arabic*),ref=(P\arabic*)]
    \setcounter {enumi}{5} \item \label{P6} The map
    $\fun : \spaz \ra \liek^*$ is continuous.
  \end{enumerate}
\end{say}
The following definition summarizes the above discussion.
\begin{defin}
  \label{def-kn}
  If $K$ is a compact connected Lie group, $G=K^\C $ and $\spaz$ is a
  topological space with a continuous $G$--action, a \emph{Kempf-Ness
    function} for $(\spaz, G,K)$ is a function
  \begin{gather*}
    \Psi : \spaz \times G \ra \R ,
  \end{gather*}
  that satisfies conditions \ref{P1}--\ref{P6}.
  We call $\fun$ the \emph{momentum mapping} of $(\spaz, G, K, \Psi)$.
\end{defin}

\begin{say}\label{notazione-varieta}
  The basic example of a Kempf-Ness function is the following.  Let
  $(M,J)$ be a compact K\"ahler manifold of complex dimension $n$ and
  let $K$ be a compact connected Lie group.  Assume that $K$ acts
  almost effectively and holomorphically on $M$ and that $g$ is a
  $K$-invariant \Keler metric with \Keler form $\om$.  If $v\in \liek$,
  let $v_M\in \XX(M)$ denote the fundamental vector field induced on
  $M$.  Assume that the action of $K$ on $(M,\om)$ is Hamiltonian and
  fix a momentum mapping
  \begin{gather*}
    \mu : M \ra \liek^*.
  \end{gather*}
  If $v\in \liek$, set $\mu^v :=\sx \mu, v\xs $. That $\mu$ is a
  momentum mapping means that it is $K$-equivariant and that
  $d\mu^v = i_{v_M}\om$.  It is well-known that the action of $K$
  extends to a holomorphic action of the complexification $G:=K^\C$.

  It has been proven by Mundet \cite[\S 3]{mundet-Crelles} that one
  can always choose a function $\Psim : M \times G \ra \R$ that has
  the properties \ref{P1}--\ref{P4}. Set
  $\psim _x(gK) : = \Psim(x, g\meno)$. Then
  \begin{gather}\label{caso-mundet}
    -d\psim_x(o)(\dga^v(0))= \desudtzero \Psim(x, \exp (itv) ) =
    \mu^v(x).
  \end{gather}
  Hence $\fun = \mu $ and \ref{P6} is satisfied. Moreover it follows
  from \eqref{cociclo-3} that
  $d\psim_x(gK) (g\dga^v (0)) = d\psim_{g\meno x}(o) (\dga^v(0))$.
  Since $M$ is compact, $||\mu|| $ is a bounded function, so $\psim_x$
  is Lipschitz. Thus also \ref{P5} holds.

  In the case of a projective manifold with the restriction of the
  Fubini-Study metric $\psi_x^M$ is the function originally defined by
  Kempf and Ness in \cite{kempf-ness}.
\end{say}

Next we deduce some immediate consequences of Definition \ref{def-kn}.
\begin{prop}\label{equivarianza}
  The map $\fun : \spaz \ra \liek^*$ is $K$-equivariant.
\end{prop}
\begin{proof}
  By the cocycle condition
  $\Psi(kx, \exp(itv)) = \Psi(x, \exp(itv)k) - \Psi(x, k)$. So using
  the left-invariance of $\Psi(x,\cdot)$ with respect to $K$ we get
  \begin{gather*}
    \begin{split}
      \sx \fun (kx),\x \xs &=\deze \Psi (x, \exp(it\x)k)=\deze \Psi
      (x,k^{-1} \exp(it\x) k)\\
      &=\deze \Psi\left (x,\exp(it\mathrm{Ad}(k^{-1})(\x)\right
      )=\mathrm{Ad}^* (k) ( \fun (x))(\x) .
    \end{split}
  \end{gather*}
\end{proof}

\begin{say}
  \label{zero} Taking $g=h=e$ in the cocycle condition \ref{P4} we
  obtain $\Psi(x, e) = 0$ and so $\Psi(x,k)=0$ for every $k\in
  K$. Moreover, for any $x\in \spaz$ and for any $g,h\in G_x$ we have
  \begin{equation}
    \label{somma}
    \Psi(x,hg)=\Psi(x,g)+\Psi(x,h).
  \end{equation}
\end{say}
\begin{lemma}
  \label{lineare-1}
  If $\x \in \liek$ and $iv\in \lieg_x$, then
  $ \Psi(x, \exp(itv))= \psi_x(\exp(-itv) K)$ is a linear function of
  $t$.
\end{lemma}
\begin{proof}
  By \eqref{somma}
  \begin{gather*}
    \Psi(x, \exp(i (t+s)v) = \Psi(x, \exp(itv)\cd \exp(isv) ) =\\
    = \Psi(x, \exp(itv)) + \Psi(x,\exp (isv)).
  \end{gather*}
  This shows that $t \mapsto \Psi(x, \exp(itv))$ is a morphism from
  $(\R, +)$ to itself. Since it is continuous, it is a linear map.
\end{proof}

\begin{say}
  The linear function considered in the previous lemma is an analogue
  of Futaki invariant in the geometry of K\"ahler-Einstein metrics,
  see e.g. \cite[Chapter 3]{tian-libro} and \cite[p. 253]{thomas-GIT}.
\end{say}

\begin{lemma}
  \label{convstab}
  The function $\psi_x$ is geodesically convex on $X$. More precisely,
  if $\x \in \liek$ and $\alf(t) = g \exp(it\x) K$ is a geodesic in
  $X$, then $\psi_x \circ \alfa$ is either strictly convex or
  affine. The latter case occurs if and only if
  $ g \exp (\C \x) g \meno \subset G_x$.  In the case $g=e$, the
  function $\psi_x \circ \alfa$ is linear if
  $ \exp (\C \x) \subset G_x$ and strictly convex otherwise.
\end{lemma}
\begin{proof}
  Fix $t_0 \in \R$. Set $h: = g \exp(it_0\x)$. By \eqref{cociclo-psi}
  \begin{equation}
    \label{sopra-2}
    \begin{gathered}
      \psi_x (\alfa(t_0 + s)) = \psi_x ( h \exp(is\x) K) 
      =\psi_x( hK ) + \psi_{h\meno x} ( \exp(is\x)K).
    \end{gathered}
  \end{equation}
  Hence
  \begin{multline*}
    \frac{\mathrm{d^2}}{\mathrm{dt}^2 } \bigg |_{t=t_0}
    \psi_x(\alfa(t)) = \\
    =\frac{\mathrm{d^2}}{\mathrm{ds}^2 } \bigg |_{s=0} \psi_{h\meno x}
    (\exp(is\x) K) = \frac{\mathrm{d^2}}{\mathrm{ds}^2 } \bigg |_{s=0}
    \Psi(h\meno x, \exp(-is\x)).
  \end{multline*}
  Therefore \ref{P3} yields convexity of $\psi_x\circ \alfa$. If
  $\psi_x\circ \alfa$ is not strictly convex at $t_0$, then by
  \ref{P3} we conclude that $\exp(\C\x) \subset G_{h\meno x}$.  By the
  previous lemma $\psi_x(\exp(itv)K)=\Psi(h\meno x , \exp (-itv))$ is
  a linear function of $t$.  By \eqref{sopra-2} we have
  $\psi_x(\alfa(t)) = \psi_x(hK) + \psi_x(\exp(i(t-t_0)v)K)$.  This
  proves $\psi_x \circ \alpha$ is affine.  Moreover from
  $\exp(\C\x) \subset G_{h\meno x}$ it follows that
  $g \exp(\C \x) g\meno \subset G_x$. The same computation shows that
  conversely if $g \exp(\C \x) g\meno \subset G_x$ then
  $\psi_x\circ \alfa$ is affine. In case $g=e$ we know that
  $\psi_x(K) =0$ by \ref{somma}, so if the function is affine, it is
  in fact linear.
\end{proof}

\begin{say}
  The group $G$ acts isometrically on $X$ from the left: for $g\in G$
  the map
 \begin{gather*}
    L_g : X \ra X , \qquad L_g(hK) :=ghK,
  \end{gather*}
  is an isometry.  We will sometimes write simply $gx$ for
  $L_g(x)$. The cocycle condition \ref{P4} is equivalent to the
  following identity between two functions and a constant:
  \begin{gather}
    \label{cociclo-3}
    L_g^* \psi_x = \psi_{g\meno x} + \psi_x(gK).
  \end{gather}
\end{say}

\begin{say}\label{say-def-la}
  Since $\psi_x $ is Lipschitz, we can apply the machinery of
  \ref{ss1}-\ref{ss2}. For $x\in \spaz$ denote by $\la_x$ the function
  $(\psi_x)_\infty$:
  \begin{gather*}
    \la_x : \tits \ra \R, \qquad \la_x ([\ga]): = \lim_{t\to +\infty}
    \desudt \psi_x(\ga(t)).
  \end{gather*}
  Set
  \begin{gather}
    \label{def-la}
    \la : \spaz \times \tits \lra \R , \quad \la(x, p):= \la_x(p).
  \end{gather}
  By Lemma \ref{convex-function} $\la_x $ and hence $\la$ are
  well-defined and finite. We call $\la_x$ the \emph{maximal weight}
  of $x$.  Using the notation defined in \eqref{def-e} for
  $v\in S(\liek)$ we have
  \begin{gather}
    \label{la-exp}
    \la_x (\e(v)) = \lim_{t\to +\infty} \desudt \psi_x(\exp(itv)K) =
    \lim_{t\to +\infty} \desudt \Psi(x, \exp(-itv)).
  \end{gather}
\end{say}

\begin{say}
  \label{say-azione-tits}
  If $g\in G$ and $\ga$ is a unit speed geodesic, then also
  $g \circ \ga$ is a unit speed geodesic and clearly
  $\ga \sim \ga' \Leftrightarrow g\circ \ga \sim g \circ \ga'$. Thus
  setting
  \begin{gather*}
    g \cd [\ga]: = [g\circ \ga]
  \end{gather*}
  defines an action of $G$ on $\tits$.  Using the bijection
  $\e : S(\liek) \ra \tits$, introduced in \eqref{def-e}, we get a
  action of $G$ on $S(\liek)$:
  \begin{gather}
    \label{azione-tits}
    g \cd \x : = \e\meno ( g\cd e(\x)).
  \end{gather}
  This action is continuous with respect to the sphere topology on
  $\tits$ (see e.g \cite[p. 41]{borel-ji-libro}), but it is not
  smooth.
\end{say}

\begin{lemma}
  \label{lemma-la-equiv}
  For any $x\in \spaz$, any $g\in G$ and any $p\in \tits$
  \begin{gather}
    \la_{g\meno x}(p) = \la_x(g \cd p).
    \label{la-equiv}
  \end{gather}
\end{lemma}
\begin{proof}
  Assume that $p=[\ga]$ for some geodesic $\ga$ in $X$. Then
  $g \cd p = [g \circ \ga]$.  By \eqref{cociclo-3} we have
  $\psi_{g\meno x} = L_{g} ^* \psi_x - \psi_x(g K)$. Since
  $\psi_x(g K)$ does not depend on $t$,
  \begin{gather*}
    \desudt \psi_{g\meno x}(\ga(t)) = \desudt \psi_x( g \cd \ga(t)).
  \end{gather*}
  The result follows immediately.
\end{proof}

\begin{lemma}\label{critical-point}
  Let $x\in \misu$. The following conditions are equivalent:
  \begin{enumerate}
  \item $g\in G$ is a critical point of $\Psi(x, \cd)$;
  \item $\fun(gx) =0$;
  \item $g\meno K$ is a critical point of $\psi_x$.
  \end{enumerate}
\end{lemma}
\begin{proof}
  Let $\x \in \liek$. Using the cocycle condition \ref{P4}, one gets
  \begin{gather*}
    \Psi(x, \exp(it\x )g ) = \Psi (x,g)+ \Psi (gx,\exp(it\x ) ).
  \end{gather*}
  Therefore
  \begin{gather}
    \label{derivata-ovunque}
    \deze \Psi(x, \exp(it\x ) g) = \deze \Psi(gx, \exp(it\x )) = \sx
    \fun (gx), \x \xs.
  \end{gather}
  Since for any $k\in K$, $\Psi(x,kg)=\Psi(x,g)$, then $\fun (gx)=0$
  if and only if $g$ is a critical point of $\Psi (x,\cdot)$ if and
  only if $g\meno K$ is a critical point of $\psi_x$.
\end{proof}

\begin{say}
  Set
  \begin{gather*}
    \fun^\x (x): = \sx \fun (x), \x \xs.
  \end{gather*}
  If $t,t_0\in \R$, then
  $\exp(i(t+t_0)\x ) = \exp (it\x) \cd \exp(it_0\x)$. Setting
  $g=\exp(it_0\x)$ in \eqref{derivata-ovunque} yields
  \begin{gather}
    \label{derivata-t}
    \begin{gathered}
      \frac{\mathrm{d}}{\mathrm{dt} } \bigg |_{t=t_0} \Psi
      (x,\exp(it\x ))=
      \deze \Psi (x,\exp(it\x) \cd \exp(it_0\x))= \\
      = \fun^\x (\exp(it_0 \x )\cdot x).
    \end{gathered}
  \end{gather}
\end{say}

\begin{say}
  \label{def-compatibile}
  Let $U$ be compact Lie group and let $U^\C$ be its complexification
  which is a reductive complex algebraic group.  The map
  $f: U \times i\liu \ra U^\C$, $f(g, \x ) = g \cd \exp \x $ is a
  diffeomorphism.  If $H\subset U^\C$ is a closed subgroup, set
  $L:=H\cap U$ and $\liep:= \lieh \cap i\liu$.  We say that $H$ is
  \enf{compatible} (\cite{heinzner-schwarz-stoetzel,
    heinzner-stoetzel-global}) if $f (L \times \liep) = H$. The
  restriction of $f$ to $L\times \liep$ is then a diffeomorphism onto
  $H$. It follows that $L$ is a maximal compact subgroup of $H$ and
  that $\lieh = \liel \oplus \liep$.  Note that $H$ has finitely many
  connected components.

\end{say}

 \begin{prop}\label{compatible}
   If $\fun (x)=0$, then $G_x$ is compatible.
 \end{prop}
 \begin{proof}
   Let $g\in G_x$. Then $g=k \exp (i\x)$ for some $k\in K$ and
   $\x\in \liek$. By Proposition \ref{equivarianza}, we have
   $\fun (\exp (i\x) x)=0$. Let $f(t):=\fun^{\x} (\exp (it\x)x)$. Then
   $f(0)=f(1)=0$ and  using \eqref{derivata-t}
   \begin{gather*}
     \desudt f(t)=\desudt \fun^{\x} (\exp
     (it\x)x)=\frac{\mathrm{d^2}}{\mathrm{dt}^2 } \Psi(x,\exp(it\x))
     \geq 0.
   \end{gather*}
   Therefore
   $\frac{\mathrm{d^2}}{\mathrm{dt}^2 } \Psi(x,\exp(it\x))=0$ for
   $0\leq t \leq 1$.  It follows from \ref{P3} that
   $\exp(\C\x) \subset G_x$, so $v \in \lieg_x \cap i \liek$ and $G_x$
   is compatible.
 \end{proof}

 \section{Stability}
 \label{sec:stab-with-resp}

 Let $(\spaz,G, K) $ be as in the previous section, let $\Psi$ be a
 Kempf-Ness function and let $x\in \spaz$.

 \begin{defin}
   \label{stabilita}
   \begin{enumerate}
   \item $x$ is \enf{polystable} if $G\cd x \cap \mume \neq \vacuo$.
   \item $x$ is \enf{stable} if it is polystable and $\lieg_x$ is
     conjugate to a subalgebra of $\liek$.
   \item $x$ is \enf{semi--stable} if
     $\overline{G\cd x} \cap \mume \neq \vacuo$.
   \item $x$ is \enf{unstable} if it is not semi--stable.
   \end{enumerate}
 \end{defin}

 \begin{say}
   The four conditions above are $G$-invariant in the sense that if a
   point $x$ satisfies one of them, then every point in the orbit of
   $x$ satisfies the same condition. This is clear from the definition
   for polystability, semi--stability and unstability. To check that
   stability is also $G$-invariant it is enough to recall that
   $\lieg_{gx} = \Ad(g) (\lieg_x)$.
 \end{say}

  \begin{lemma}\label{intersezione-nulla}
    If $\lia \subset \lieg$ is a subalgebra which is conjugate to a
    subalgebra of $\liek$, then $\lia \cap i\liek =\{0\}$.
  \end{lemma}
  \begin{proof}
    It is enough to show that $\Ad (g)( \liek) \cap i \liek =\{0\}$
    for any $g\in G$.  Choose an embedding
    $j:G\hookrightarrow \mathrm{GL}(n,\C)$ such that
    $j(K) \subset\mathrm{U}(n)$.  If $\x \in i\liek$ and
    $\x = \Ad (g) w$ with $w \in \liek$, then $ j(\x )$ is a Hermitian
    matrix, hence with real eigenvalues, while $j(w)$ is skew-Hermitan
    with imaginary eigenvalues. Since these matrices are similar the
    only eigenvalue is 0. Thus $j(\x ) =0$ and $\x =0$.
  \end{proof}

\begin{teo}\label{stabile}
  The following conditions are equivalent: (1) $x\in \spaz$ is stable,
  (2) $\la_x >0 $ on $\tits$, (3) $\psi_x$ is an exhaustion function.
\end{teo}
\begin{proof}
  (1) $\Rightarrow $ (2). Assume that $\fun(gx) =0$ for some $g\in G$.
  By Lemma \ref{critical-point}, $g$ is a critical point of
  $\Psi(x, \cd)$. Set $y=g x$.  We start by proving that
  $\lambda_y   >0$ on $\tits$. Using \eqref{la-exp} and the fact that 
  $\Psi(y, \cd)$ is a convex function we get
  \begin{gather*}
    \lambda_y (\e(-\x) )\geq \desudtzero \Psi(y, \exp(it\x ) )= \sx \fun
    (y), \x \xs=0.
  \end{gather*}
  If $\lambda_y (\e(-\x) )=0$ for some $\x \in S(\liek)$, then
  $\frac{\mathrm{d^2}}{\mathrm{dt}^2 } \Psi(x,\exp(it \x ))=0$ for any
  $t \geq 0$. Using \ref{P3} it follows that
  $\exp(\C \x ) \subset G_y$, so $i\x \in \lieg_y\cap i\liek$.  Since
  $x$ is stable, $\mathfrak g_y = \Ad(g) (\lieg_x)$ is conjugate to a
  subalgebra of $\liek$, so Lemma \ref {intersezione-nulla} implies
  that $\x =0$.  Since we are assuming $\x \in S(\liek)$ this is
  absurd. Therefore $\la_y >0$ as desired.  By Lemma
  \ref{lemma-la-equiv}
  $ \la_x(p) = \la _{g\meno y } (p) = \la_y (g \cd p ) $. Thus
  $\la_x >0 $ on $\tits$.  By \ref{P5} $\psi_x$ is Lipschitz
  continuous. So (2) $\Leftrightarrow$ (3) is the content of Lemma
  \ref{convex-function}.  Finally we prove that (3) $\Rightarrow $
  (1). Since $\psi_x : X=G/K \ra \R$ is an exhaustion, there is a
  minimum point $g K \in X$. Set $y:=gx$. Thus
  $y \in G\cd x \cap \mume$ and $x$ is polystable. To complete the
  proof we need to show that $\lieg_y \subset \liek$.  Let
  $i\x \in \lieg_y \cap i\liek$.  By Lemma \ref{lineare-1}
  $ \Psi(y, \exp(it\x ))$ is linear function of $t$.  On the other
  hand (3) implies that this function is an exhaustion, if
  $\x \neq 0$. Since a linear function cannot be an exhaustion, we
  conclude that $\x =0$. This proves that
  $\lieg_y \cap i\liek = \{0\}$.  By Proposition \ref{compatible} the
  stabilizer $G_y$ is compatible. In particular
  $\lieg_y = (\lieg_y \cap \liek) \oplus( \lieg_y \cap i \liek)$.  So
  $\mathfrak g_y \subset \liek$. This proves that $x$ is stable.
\end{proof}

\begin{cor}
  \label{stabcomp}
  If $x\in \spaz$ is stable, then $G_x$ is compact.
\end{cor}
\begin{proof}
  Let $g\in G$ be such that $\fun(gx)=0$ and set $y=gx$. By
  Proposition \ref{compatible} $G_y$ is compatible, so has only
  finitely many connected components. Moreover $G_y^0$ is compact
  since $\lieg_y \subset \liek$ as shown in the previous proof. It
  follows that $G_y$ and $G_x = g\meno G_y g $ are both compact.
\end{proof}

\begin{say}\label{sotto-struttura}
  If $\spaz'$ is a $G$-invariant subspace of $ \spaz $, the
  restriction of $\Psi $ to $G\times \spaz'$ is a Kempf-Ness function
  for $(\spaz', G, K)$. The functions $\la$ and $\fun$ for
  $(\spaz', G, K)$ are simply the restrictions of those for $\spaz$.
\end{say}

\begin{say}\label{sotto-struttura-2}
  If $K'\subset K$ is a closed subgroup and $G':= (K')^\C$, there are
  totally geodesic inclusions $X':= G'/K' \hookrightarrow X$ and
  $\partial_\infty X' \subset \tits$.  If $\Psi$ is a Kempf-Ness
  function for $(G,K, \spaz)$, then
  $ \Psi^{K'}:=\Psi\restr{\spaz\times G'}$ is a Kempf-Ness function
  for $(G', K', \spaz)$.  The related functions are
  \begin{gather*}
    \fun^{K'}: \spaz \ra {\liek'}^* , \qquad \fun^{K'}(x):=
    \fun(x)\restr
    {\liek'},\\
    \psi^{K'}_x := \psi _x \restr{X'}, \qquad \la^{K'} = \la
    \restr{\spaz\times \tits '}.
  \end{gather*}
\end{say}

The following Corollary is analogous to the stability part in the
Hilbert-Mumford criterion.
\begin{cor}\label{Hilbert-Mumford}
  A point $x \in \spaz$ is $G$-stable if and only if it is
  $T^\C$-stable for any compact torus $T \subset K$.
\end{cor}
\begin{proof}
  Assume that $x$ is $G$-stable. Then $\la_x >0$. So for any torus
  $T\subset G$, we have $\la^T_x>0$. By Theorem \ref{stabile} $x$ is
  $T^\C$--stable. Conversely assume that $x$ is $T^\C$--stable for any
  compact torus $T\subset K$. If $\x \in S(\liek)$ choose a torus $T$
  such that $\x \in \liet$.  Then $\la_x(\e(\x)) = \la^T_x(\e(\x)) $
  and $\la^T_x(\e(\x)) >0$ by Theorem \ref{stabile}, since $x$ is
  $T^\C$--stable. Hence $\la_x>0$ on $\tits$. Using again Theorem
  \ref{stabile} we conclude that $x$ is $G$-stable.

\end{proof}

\begin{lemma} [\protect{\cite[Prop. 3.11
    (5)]{teleman-symplectic-stability}}]
  \label{semicontinua}
  The function $ \la : \spaz \times \tits \lra \R $ is lower
  semicontinuous if $\tits $ is endowed with the sphere topology (see
  \ref{ss1}).
\end{lemma}
\begin{proof}
  Fix $(x, p) \in \spaz\times \tits$ with $p=\e(v)$ and
  $v\in S(\liek)$. We need to show that for any $\eps>0$, there is a
  neighbourhood $A$ of $(x, v)$ in $\spaz\times S(\liek)$, such that
  for $(x', v')\in A$ we have $\la(x', \e(v')) > \la(x, p) - \eps$.
  By 
  \eqref{la-exp} and \eqref{derivata-t}
  \begin{gather*}
    \la_x (\e(v)) = \lim_{t \to +\infty} \fun^v (\exp(-itv) \cd x).
  \end{gather*}
  So given $\eps > 0$, there is a $t_0\in \R$ such that
  \begin{gather*}
    \fun^ \x ( \exp(-it_0\x) \cd x) > \la_x(p) -\eps.
  \end{gather*}
  By the continuity of the $G$--action on $\spaz$ and \ref{P6} the
  function
  \begin{gather*}
    f: \spaz\times S(\liek) \lra \R, \quad f (x', v') : =
    \fun^{v'}(\exp(-it_0v')\cd x')
  \end{gather*}
  is continuous.  Since $f(x, v) > \la_x(p) - \eps$ there is a
  neighbourhood $A$ of $(x, v) $ in $\spaz \times S(\liek)$ such that
  $f(x', v') > \la_x(p) - \eps$ for any $(x', v') \in A$.  Since the
  function $t \mapsto \fun^{\x'}(\exp( -it\x' )\cd x) $ is increasing,
  we get $\la(x', \e(v')) > \la_x(p) - \eps $ for any
  $(x', v') \in A$, as desired.
\end{proof}

\begin{cor}\label{stable-open-abstract-setting}
  The set of stable points is open in $\spaz$.
\end{cor}
\begin{proof}
  Let $\pi: \spaz \times \tits \ra \spaz$ be the projection on the
  first factor. Since $\tits$ is compact in the sphere topology and
  $\spaz$ is Hausdorff, $\pi$ is a closed map \cite[Thm. 2.5,
  p. 227]{dugundji}.  The set
  $A:=\{(x, p) \in \spaz \times \tits: \la(x, p) >0\}$ is open in
  $\spaz\times\tits$ since $\la$ is lower semicontinuous.  So its
  complement $E$ is closed and $\pi(E)$ is closed in $\spaz$.  By
  Theorem \ref{stabile} the set of stable points is the complement of
  $\pi(E)$. Therefore it is open.
\end{proof}

\section{Polystability and semi-stability}
\label{sec:polyst-semi-stab}

\begin{say}
  The aim of this section is to characterize polystability and
  semi-stability of $x \in \spaz$ in terms of the maximal weight
  $\la_x$.  In the classical case of a Hamiltonian action on a
  K\"ahler manifold these characterizations are due to Mundet i Riera
  \cite{mundet-Trans} and Teleman \cite {teleman-symplectic-stability}
  respectively. Our results are more general, since we deal with a
  Kempf-Ness function on a topological space. Nevertheless our
  hypothesis are stronger from another point of view: \ref{P5} is
  stronger than Assumption 1.2 in \cite{mundet-Trans}, while in the
  treatment of semi-stability we assume $\spaz$ compact, which is
  stronger than energy completeness (see
  \cite[Def. 3.8]{teleman-symplectic-stability}).

  We start with some technical lemmata.
\end{say}

\begin{lemma}\label{ziocantante}
  If $\x, \x'\in \liek$ commute and $i\x ' \in \lieg_x$ then
  \begin{multline*}
    \lim_{t\to +\infty} \desudt \Psi(x, \exp(it(\x+\x')) ) = \\
    =\lim_{t\to +\infty} \desudt \Psi(x, \exp(it\x) ) + \lim_{t\to
      +\infty} \desudt \Psi(x, \exp(it\x') ) .
  \end{multline*}
\end{lemma}
\begin{proof}
  Using the commutativity and the cocycle condition \ref{P4}
  \begin{gather*}
    \Psi(x, \exp(it(\x+\x')) ) = \Psi(x, \exp(it\x) \cd \exp(it\x') )
    \\
    = \Psi(\exp(it\x')\cd x, \exp(it\x))
    + \Psi( x, \exp(it\x')) \\
    = \Psi(x, \exp(it\x)) + \Psi( x, \exp(it\x')).
  \end{gather*}
  The result follows immediately.
\end{proof}
The previous lemma corresponds to \cite[Lemma 3.8]{mundet-Trans}. Note
that there one should assume $iv'\in \lieg_x$, which is the only case
needed later in that paper.

\begin{cor}
  \label{stabilizzatore}
  If $\la_x \geq 0$, then $\Psi(x, \exp(i\x) )=0$ for any
  $\x\in \liek$ such that $i\x \in \lieg_x$.
\end{cor}
\begin{proof}
  The result is obvious if $v=0$, see \ref{zero}.  If $v\neq 0$, we
  can assume by rescaling that $v\in S(\liek)$. Set
  $f(t) : = \Psi(x, \exp(it\x))$.  By Lemma \ref{lineare-1}
  $f(t) = at$ for some $a\in \R$.  Using \eqref{la-exp} we get
  $\la(\e(-v)) = a$ and $ \la_x(\e(\x))=-a$.  Since $\la_x \geq 0$,
  $a=0$ and $f\equiv 0$.
\end{proof}

\begin{lemma}
  \label{centralizzante}
  Let $x$ be a point in $ \spaz$.  If $\sx \fun (x) , \x\xs =0 $ for
  any $\x\in \liez_{\liek}(\liek_x)$, then $\fun (x) =0$.
\end{lemma}
\begin{proof}
  Fix an $\Ad$--invariant scalar product on $\liek$ and identify
  $\fun(x)$ with a vector in $\liek$. It is to enough to show that
  $\fun(x)$ belongs to $\liez_{\liek}( \liek_x)$.  See
  \cite[Rem. 2.13]{teleman-symplectic-stability} for the rest of the proof.
\end{proof}

\begin{say}
  \label{alfa-infinito}
  Given a geodesic $\alfa$ in the symmetric space $X$, denote by
  $\alfa(+\infty)$ the equivalence class of $\alfa$ and by
  $\alfa(-\infty)$ the equivalence class of the geodesic
  $t \mapsto \alfa(-t)$.
\end{say}

\begin{defin}
  We say that $p$ and $p' \in \tits$ are \enf {connected} if there
  exists a geodesic $\alfa$ in $X$ such that $p=\alfa(\infty)$ and
  $p'=\alfa(-\infty)$.
\end{defin}
For $x\in \spaz$ set
\begin{equation}
  \label{eq:2}
  Z(x):=\{p\in \tits : \la_x(p) =0\}.
\end{equation}

  \begin{lemma}\label{polistabile-tits}
    If $\fun(x) =0$, then $\lieg_x = \liek_x \oplus i \lieq $ with
    $\lieq\subset \liek_x$ and $Z(x) = \e (S(\lieq))$.
  \end{lemma}
  \begin{proof}
    By Proposition \ref{compatible} $G_x$ is a compatible subgroup of
    $G$, so $\lieg_x = \liek_x \oplus \liep$ with
    $\liep \subset i\liek$. Set $\lieq : = i\liep$. From Lemma
    \ref{lineare-1} and \ref{P3} it follows that for $v\in \liek$ the
    condition $iv \in \lieg_x$ is equivalent to
    $ \C v \subset \lieg_x$.  So $\lieq \subset \liek_x$.  To prove
    the last assertion fix $ v \in S(\liek)$. If $\e(v ) \in Z(x)$,
    then $f(t) := \psi_x(\exp(itv)K)$ is convex and satisfies
    $f'(+\infty) = \la_x(\e(v)) = 0$ and $f'(0) = \fun^v(x) = 0$.
    Hence $f$ is constant for $t>0$, so $iv \in \lieg_x$ by \ref{P3}
    and $v \in \lieq$. This proves that
    $\e \meno (Z(x)) \subset S(\lieq)$.  Conversely, if
    $v \in S(\lieq)$, then $iv \in \lieg_x$, so $f$ is linear by Lemma
    \ref{lineare-1}. Moreover $f'(0) = \fun^v(x) = 0$. So $f \equiv 0$
    and $\e(v) \in Z(x)$. 
  \end{proof}

  \begin{say}
    \label{def-tu}
    For $u\in \liek$ denote by $T_u$ the closure of $\exp (\R u)$ in
    $K$ and denote by $K^u$ the centralizer of $u$ in $K$, i.e.
    \begin{gather*}
     K^u:=\{ a\in K: \Ad a (u) = u\}.  
    \end{gather*}
    Similarly $G^u$ is the centralizer in $G$. Note that
    $G^u=(K^u)^\C$.
  \end{say}

  \begin{lemma} \label{mundet1} If $g\in G$ and $u\in S(\liek)$, we
    have $\dim T_u = \dim T_{g\cd u}$.
  \end{lemma}
  The action of $G$ on $S(\liek)$ is defined in \eqref{azione-tits}.
  The proof can be found in \cite[Lemma 2.1]{mundet-Trans}.

\begin{lemma}\label{lemmozzo}
  Let $x \in \spaz$ and assume that $\la_x \geq 0$.  Let
  $u \in \e\meno(Z(x))$ be such that
  \begin{gather}
    \label{massimo}
    \dim T_u = \max_{w \in \e\meno(Z(x))} \dim T_w.
  \end{gather}
  Let $K'\subset K^u$ be a compact connected subgroup such that the
  morphism
  \begin{gather}
    \label{rivestimento}
    T_u \times K' \ra K^u, \quad (a,b) \mapsto ab,
  \end{gather}
  is surjective and with finite kernel.  Set $G':= (K')^\C$.  If
  $\fun^u (x) =0$, then $x$ is $G'$-stable.
\end{lemma}
\begin{proof}
  First of all we claim that $iu \in \lieg_x$. In fact
  $\la_x(\e(u)) = 0$ and $\fun^u(x) = 0$ by hypothesis. So the convex
  function $f(t) = \psi_x(\exp(itu)K)$ satisfies
  $f'(0)=\lim_{t\to +\infty} f'(t)=0$. It follows that $f$ is constant
  on $[0,+\infty)$. Using \ref{P3} we conclude that
  $\exp(\C u) \subset G_x$.  This proves the claim.  Next set
  $X' := G' / K' $. Then $X' \hookrightarrow X=G/K$ since
  $G' \cap K =K'$.  As noted in \ref{sotto-struttura-2},
  $\Psi \restr{\spaz\times G'}$ is a Kempf-Ness function for
  $(\spaz, G', K')$.  We claim that $Z(x) \cap \tits' = \vacuo$ and we
  argue by contradiction.  Assume that there is $u' \in S(\liek')$,
  such that $\e(u') \in Z(x) \cap \tits'$.  Let $a>0$. Since
  $[u', u]=0$ and $iu\in \lieg_x$, Lemma \ref{ziocantante} yields
  \begin{gather*}
    \lim_{t\to +\infty} \desudt \Psi (x, \exp(- it (au' + u) ) = \\
    =a \lim _{t\to +\infty} \desudt \Psi (x, \exp(- it u'))
    + \lim _{t\to +\infty} \desudt \Psi (x, \exp(- it u))= \\
    = a \la_x (\e(u')) + \la_x(\e(u)).
  \end{gather*}
  By assumption $\la_x(\e(u)) =\la_x(\e(u'))=0$.  It follows that for
  any $a\in \R$ the vector $ ({u + au'})/{|u+au'|} $ belongs to
  $ \e \meno (Z(x)) $ We claim that for some $a$
  $ \dim T_{u+au'} >\dim T_u$.  Let
  $T'= \overline{ \exp (\R u+ \R u' )} $ and
  $T_{u'}=\overline{\exp(\R u')}$. Since $T_{u'}\subset K'$ and
  $K'\cap T_u$ is finite, the morphism
  \begin{gather*}
    f: T_u \times T_{u'} \longrightarrow T', \qquad f (a,b)= ab,
  \end{gather*}
  is a finite covering.  Let $\{e_1,\ldots,e_n\}$ (respectively
  $\{e_1',\ldots,e_m'\}$) be a basis of the lattice
  $\ker \exp \subset \mathfrak t_u$ (respectively
  $\ker \exp \subset \mathfrak t_{u'}$).  If
  $u=X_1 e_1+\cdots +X_n e_n$ and $u'=Y_1 e_1' +\cdots + Y_m e_m'$,
  then
  $u+au'=X_1 e_1+\cdots +X_n e_n + aY_1 e_1' +\cdots + aY_m e_m'$.
  Denote by $T'_{u+au'}$ the closure of $\exp(\R (u+au'))$ in $T'$.
  Since $f$ is a covering $\dim T_{u+au'} = \dim T' _{u+au'}$.
  Therefore
  \begin{gather*}
    \dim T_{u +a u'} = \dim_{\mathbb Q} \left(\mathbb Q X_1 + \cdots
      +\mathbb Q X_n +\mathbb Q \, aY_1 +\cdots +\mathbb Q \, a
      Y_m\right).
  \end{gather*}
  (See e.g. \cite[p. 61]{duistermaat-kolk-Lie}.)  Since $u'\neq 0$,
  $Y_j \neq 0$ for some $j$. Choose $a$ such that
  $aY_j \not \in \mathbb Q X_1 + \ldots +\mathbb Q X_n$ and set
  $w:=(u + au')/|u +au'|$.  Then $\e(w) \in Z(x)$ and
  $\dim T_{w}
= \dim T_{u + a u'}
 > \dim T_u$. This contradicts \eqref{massimo}.  We have
  proved that $Z(x) \cap \tits'=\vacuo$. Since $\la_x \geq 0$ on
  $\tits$, we conclude that $\la_x >0 $ on $\tits'$. By Theorem
  \ref{stabile} $x$ is $G'$-stable.
\end{proof}

\begin{teo}[Mundet i Riera]
  \label{ps}
  A point $x\in \spaz$ is polystable if and only if $\la_x \geq 0$ and
  for any $p\in Z(x)$ there exists $p' \in Z(x)$ such that $p$ and
  $p'$ are connected.
\end{teo}

\begin{proof}
  Let $x \in \spaz$ be a point satisfying the condition in the
  theorem.  If $Z(x) =\vacuo$, then $\la_x >0$, so by Theorem
  \ref{stabile} $x$ is stable and a fortiori polystable.  If
  $Z(x)\neq \vacuo$, choose $p\in Z(x)$ such that the dimension of the
  torus $T_\x$, where $\x:= \e\meno (p)$ is the largest possible.  In
  other words
  \begin{gather*}
    \dim T_\x = \max_{w \in \e\meno(Z(x))} \dim T_w.
  \end{gather*}
  By assumption there is a geodesic $\alfa$ in $X$ such that
  $p=\alfa(+\infty)$ and $p' = \alfa(-\infty) \in Z(x)$, using the
  notation of \ref{alfa-infinito}.  Assume that
  $\alfa (t) =g \exp(itu) K$. Then $p = g\cd \e(u) $ and
  $p' = g \cd \e(-u)$. By \eqref{la-equiv}
  \begin{gather*}
    \la_{g\meno\cd x}(\e(u)) = \la_x (g \cd \e(u)) = \la_x (p) = 0,\\
    \la_{g\meno\cd x}(\e(-u)) = \la_x (g \cd \e(-u)) = \la_x (p') = 0.
  \end{gather*}
  Set $y:= g\meno \cd x$. We have just proved that the convex function
  $t\mapsto \psi_y (\exp(itu)K)$ has zero derivative at both $+\infty$
  and $-\infty$. So it is constant and by \ref{P3}
  $\exp (\C u) \subset G_y$.  Moreover $\fun^u (y) =0$ by
  \eqref{momento-astratto}.  Since $\e(\x) = p =g\cd \e(u)$, Lemma
  \ref{mundet1} implies that $\dim T_u = \dim T_\x $.  From
  \eqref{la-equiv} we deduce that $Z(x) = g ( Z(y))$ and that
  $\la_y \geq 0$.  . Using again Lemma \ref{mundet1} we get
  \begin{gather*}
    \dim T_u = \max_{w \in \e\meno(Z(y))} \dim T_w.
  \end{gather*}
  Fix a compact connected subgroup $K' \subset K^u $ such that the map
  \eqref{rivestimento} is a finite covering.  Set $G':= (K') ^\C$.  By
  Lemma \ref{lemmozzo} $y \in \spaz$ is $G'$-stable.  Denote by
  $\fun': \spaz \ra \liek'^*$ the momentum mapping of
  $(\spaz, G', K' , \Psi\restr{\spaz\times G'})$. Thus there is
  $h\in G' \subset G$ such that $\fun'(h\cd y)=0$.  Set $z:=h\cd y$.
  The fact that $\fun'(z)=0$, means that $\sx \fun(z) , v\xs =0$, for
  any $v \in \liek'$, see \ref {sotto-struttura-2}.  We claim that
  $\sx \fun(z), v\xs =0$ also for $v\in \liet_u$.  First we prove that
  $\liet_u \subset \liek_z $.  In fact let $\spaz' $ denote the set of
  points of $\spaz$ that are fixed by $\overline{\exp(\C u)}$.  Note
  that $y\in \spaz'$.  Since $G^u$ preserves $\spaz'$, also
  $z \in \spaz'$, i.e. $ \overline{\exp(\C u)} \subset G_z$, as
  desired.  Now if $v\in \liet_u$, then $v \in \liek_{z}$, so by Lemma
  \ref{convstab} the function $f(t)= \Psi(z,\exp(it v))$ is
  linear. Since $ f'(-\infty) =\la_{z} ( \e(v)) \geq 0$ and
  $f'(+\infty)=\la_{z} (\e (-v)) \geq 0$, we conclude that
  $ f'(0)= \sx \fun(z) , v\xs =0$ for any $v\in \liet_u$, as
  claimed. Since $\liek'\oplus \liet_u = \liek^u$, we have proved that
  $\sx \fun(z) , v\xs =0$ for any $v\in \liek^u$.  But
  $\liet_u \subset \liek_z$, so
  $\liez_\liek (\liek_z) \subset \liek^u$.  Therefore Lemma
  \ref{centralizzante} implies that $\fun(z)=0$.  This finally proves
  that $x$ is polystable.  We have shown that the condition in the
  theorem implies polystability.  To prove the opposite implication, assume
  that $x$ is polystable. Then there exists $g\in G$ such that
  $\mathfrak F (g\cdot x)=0$. Set $y=g\cdot x$ and fix $\x \in
  \liek$. The function $ f(t)=\Psi(y,\exp(-it \x ))$ is convex and
  \begin{gather*}
    \lambda_y ( \e( \x) )= \lim_{t\to \infty} f'(t)\geq f'(0) = \sx
    \fun(y), \x \xs = 0.
  \end{gather*}
  This proves that $\la_y \geq 0$ on $\tits$. By Lemma \ref
  {lemma-la-equiv} also $\la_x \geq 0$. Next we check the condition on
  $Z(y)$.  By Proposition \ref{compatible} and Lemma
  \ref{polistabile-tits}, $G_y$ is compatible with
  $\mathfrak g_y=\mathfrak k_y \oplus \mathfrak q$ and
  $Z(y)=e(S(\mathfrak q))$.  If $\e(v) \in Z(y)$, then also
  $\e(-v) \in Z(y)$.  Since $\e (v)$ and $\e(-v)$ are obviously
  connected, the condition in the theorem holds for $Z(y)$.  Moreover
  $Z(x) = g\meno (Z(y))$, and so $Z(x)$ also satisfies the condition
  in the theorem. Indeed, if $p \in Z(x)$, then $g \cd p \in Z(y)$, so
  there is $q \in Z(y)$ connected to $g\cd p$, by a geodesic $\alfa$,
  i.e.  $\alfa(+\infty) = g\cd p$ and $\alfa(-\infty ) = q$.  Then the
  geodesic $g\meno\circ \alfa $ connects $p$ to
  $g\meno \cd q \in Z(x)$.
\end{proof}

\begin{ex}
  If $X$ is the unit disc, any two distinct points of $\tits$ are
  connected. Thus in this case the condition in the theorem means that
  either $Z(x)$ is empty, or it contains at least two points.  If
  $X=\R^n$, and $v\in S(\liek) = S^{n-1}$, then $\e(v)$ is connected
  only to $\e(-v)$. In this case the condition in the theorem amounts
  to saying that $\e\meno(Z(x) ) \subset S^{n-1}$ is invariant by the
  antipodal map.
\end{ex}

The following important lemma is due to Kapovich, Leeb and Millson
\cite [Lemma 3.4]{kapovich-leeb-millson-convex-JDG}.

\begin{lemma}
  \label{kapo}
  Let $X$ be a symmetric space and let let $u: X \ra \R$ be a smooth
  convex function. If $u_\infty \geq 0$ on $\tits$, then there is a
  sequence $\{p_n\}$ in $ X$ such that
  $|| \nabla u_\infty (p_n) || \to 0$.
\end{lemma}

The next theorem uses ideas from the proof of Theorem 4.3 in
\cite{teleman-symplectic-stability}, with simplifications due to the
previous lemma and the compactness hypothesis.

\begin{teo}\label{semi-stable-abstract}
  If $\spaz $ is compact, then a point $x \in \spaz$ is semi-stable if
  and only if $\la_x \geq 0$.
\end{teo}
\begin{proof}
  Assume that $x \in \spaz$ and that $\la_x\geq 0$. By Lemma
  \ref{kapo} there is a sequence of points $\{p_n\} \subset X$ such
  that $d\psi_x(p_n) \to 0$.  Write $p_n = g_n K$. Using \eqref
  {cociclo-3} we get $ d(L_{g_n}^* \psi_x) = d \psi_{g_n\meno x} $. So
  \begin{gather*}
    d \psi_{g_n\meno x} (o ) = d\psi_ x (g_n K) \circ dL_{g_n} (o) =
    d\psi_ x (p_n) \circ dL_{g_n} (o).
  \end{gather*}
  Since $L_g $ is an isometry of $X$,
  \begin{gather*}
    ||\fun (g_n\meno x) || = || d\psi_{g_n \meno x} (o) || = ||
    d\psi_x (p_n)|| \to 0.
  \end{gather*}
  Since $\spaz$ is compact, we can assume that $g_n\meno x $ converges
  to some point $y\in \overline {G\cd x}$. By \ref{P6} we get
  immediately that $\fun(y)=0$. So $x$ is semi--stable.  Conversely,
  assume that $x$ is semi-stable, i.e. there is
  $y\in \overline{G\cd x} \cap \fun\meno(0)$.  Pick a net of elements
  $g_\alf \in G$ such that $g_\alf x \to y$ (if $\spaz$ satisfies the
  first countability axiom one can just take a sequence).  Assume by
  contradiction that there is some $p\in\tits$ such that $\la_x(p) <0$
  and set $v:=\e\meno(p) \in S(\liek)$. Since $\fun (y) =0$, $y$ is
  polystable, so by Theorem \ref{ps} $\la_y \geq 0$.  Write
  $g_\alf \cd p = \e (\x_\alf)$ for $\x_\alf \in S(\liek)$. By passing
  to a subnet we can assume that $\x_\alf \to w$. Set $q = \e(w)$.
  Then $g_\alf \cd p \to q$ in the sphere topology. By Lemma
  \ref{semicontinua} we have
  \begin{gather*}
    \la_y (q) \leq \liminf_{\alfa} \la_{g_\alfa \cd x} (g_\alf \cd p)
    .
  \end{gather*}
  But $\la_{g_\alf \cd x} (g_\alf \cd p) = \la_x (p) < 0$ and
  $\la_y (q) \geq 0$. Thus we get a contradiction. This proves that
  $\la _x \geq 0$.

\end{proof}

\section{Measures}
\label{sez-misure}

\begin{say}
  If $M$ is a compact manifold, denote by $\misu(M)$ the vector space
  of finite signed Borel measures on $M$.  These measures are
  automatically Radon \cite[Thm. 7.8, p. 217]{folland-real-analysis}.
  Denote by $C(M)$ the space of real continuous function on $M$. It is
  a Banach space with the $\sup$--norm.  By the Riesz Representation
  Theorem \cite[p.223]{folland-real-analysis} $\misu(M)$ is the
  topological dual of $C(M)$.  The induced norm on $\misu(M)$ is the
  following one:
  \begin{gather}
    \label{totvar}
    ||\nu||:= \sup \Bigl \{ \int_M f d\nu: f\in C(M), \sup_M |f| \leq
    1 \Bigr\}.
  \end{gather}
  We endow $\misu(M)$ with the weak-$*$ topology as dual of $C(M)$.
  Usually this is simply called the \emph{weak topology} on measures. We use
  the symbol $\nu_\alf \weak \nu$ to denote the weak convergence of
  the net $\{\nu_\alf\}$ to the measure $\nu$.  Denote by
  $\proba(M) \subset \misu(M)$ the set of Borel probability measures
  on $M$.  We claim that $\pb$ is a compact convex subset of
  $\misu(M)$. Indeed the cone of positive measures is closed and
  $\proba(M)$ is the intersection of this cone with the closed affine
  hyperplane $\{\nu\in \misu(M) : \nu(M) = 1\}$. Hence $\proba(M)$ is
  closed.  For a positive measure $|\nu|=\nu$, so $\proba(M)$ is
  contained in the closed unit ball in $\misu(M)$, which is compact in
  the weak topology by the Banach-Alaoglu Theorem
  \cite[p. 425]{dunford-schwartz-1}.  Since $C(M)$ is separable, the
  weak topology on $\pb$ is metrizable \cite
  [p. 426]{dunford-schwartz-1}.
\end{say}

\begin{say}
  If $f : X \ra Y$ is a measurable map between measurable spaces and
  $\nu $ is a measure on $X$, the \emph{image measure} $f\pf\nu$ is
defined by  $f\pf\nu(A) : = \nu (f\meno A)$.  It
  satisfies the \enf{change of variables formula}
  \begin{gather}
    \label{eq:pushforward}
    \int_Y u (y) d( f\pf\nu)(y) = \int_X u(f(x)) d\nu(x).
  \end{gather}
\end{say}

\begin{lemma}
  \label{prob-continua} Let $M$ be a compact manifold.  If $G$ is a
  Lie group acting continuously on $M$, the map
  \begin{gather}
    \label{azione}
    G\times \pb \ra \pb, \quad (g, \nu) \mapsto g_*\nu,
  \end{gather}
  defines a continuous action of $G$ on $\proba (M) $ provided with
  the weak topology.
\end{lemma}
\begin{proof}
  The map is obviously an action. To check the continuity let
  $g_\alf\to g$ and $\nu _\alfa \weak \nu$ be converging nets in $G$ and
  $\pb$ respectively (since both are metrizable spaces considering
  sequences would be enough). Fix a distance $d$ on $M$ inducing the
  manifold topology.  We claim that $g_\alf \to g$ uniformly on
  $(M, d)$. Given $\eps>0$ the set
  $A:=\{(h,x): d(h\cd x, g\cd x) <\eps\}$ is open in $G\times M$.
  Since $\{g\}\times M \subset A$ and $M$ is compact, there is a
  neighbourhood $U$ of $g$ such that $U\times M \subset A$.  There
  is $\alf_0$ such that $g_\alf\in U$ for $\alf \geq \alf_0$, so
  \begin{gather*}
    \sup_{x\in M} d (g_\alf \cd x , g\cd x) < \eps, \quad \text{for }
    \alfa\geq \alf_0.
  \end{gather*}
  This proves the claim.  Since any $\phi\in C(M)$ is uniformly
  continuous, it follows that $\phi \circ g_\alf \to \phi \circ g $
  uniformly on $M$.  To prove the continuity of the action we need to
  show that $g_\alfa \pf \nu_ \alf \weak g\pf \nu$, i.e.
  \begin{gather*}
    \int_M \phi \, d(g_\alf \pf \nu_\alf ) \to \int_M \phi\, d (g\pf
    \nu), \quad \text{for any } \phi \in C(M).
  \end{gather*}
  In fact
  \begin{gather*}
    \Big | \int _M \phi \, d(g_\alf \pf \nu_\alf ) - \int_M \phi \, d
    (g\pf \nu) \Big | = \Big | \int _M \phi \circ g_\alf \, d \nu_\alf
    -
    \int_M \phi\circ g \, d\nu \Big | \leq \\
    \leq \Big| \int _M \phi \circ g_\alf \, d \nu_\alf - \int_M \phi
    \circ g \, d\nu _\alf \Big | + \Big | \int _M \phi \circ g \, d
    \nu _\alf - \int_M \phi \circ g\, d\nu \Big| .
  \end{gather*}
  Since $\nu_\alf(M) =1$, the first term is bounded by
  $|| \phi\circ g_\alfa - \phi \circ g||$, so tends to $0$. The second
  term tends to $0$, because $\nu_\alf \weak \nu$.
\end{proof}

\begin{say}
  The map \eqref{azione} is not continuous if $\pb$ is endowed with
  the topology coming from the norm \eqref{totvar}. For example if
  $g_n \to g $ in $ G$, but $g_n \cd x \neq g\cd x$ for some point
  $x\in M$, then, denoting by $\delta_x$ the Dirac measure supported
  at $x$, we have $g_n\pf \delta_x = \delta_{g_n \cd x}$ and
  $||g_n \pf \delta_x - g \pf \delta_x||=2$.
\end{say}

\begin{lemma}\label{umpfi}
  Let $X$ be a vector field on $M$ with flow $\{\phi_t\}$. If
  $ \nu \in \misu(M)$ and $X$ vanishes $\nu$--almost everywhere, then
  $\phi_t\pf \nu =\nu$ for any $t$.
\end{lemma}
\begin{proof}
  Set $N:=\{x\in M : X(x) \neq 0\}$. Then $\nu(N)=0$ and for any
  $t\in \R$ and any $x\not \in N$, $\phi_t(x)=x$. In particular both
  $N$ and $M-N$ are $\phi_t$-invariant. If $A\subset M$ is measurable,
  then
  \begin{gather*}
    \phi_{-t}(A) =\phi_{-t}((A\setminus N) \sqcup (N\cap
    A))=(A\setminus N)\sqcup \phi_{-t} (N\cap A).
  \end{gather*}
  Since $\phi_{-t} (N\cap A) \subset N$,
  $\phi_t\pf\nu(A) = \nu (\phi_{-t}( A) ) =\nu(A\setminus N)=\nu(A)$.
\end{proof}

\begin{say}
  In the sequel for $ g \in G$ and $ \nu \in \pb$, we will use the
  notation
  \begin{gather*}
    g \cd \nu := g\pf \nu.
  \end{gather*}
\end{say}

\begin{say}
  From now on we assume that $M$ is a compact K\"ahler manifold, that $K$ is a
  compact connected Lie group acting on $M$ in a Hamiltonian fashion
  with momentum mapping $\mu$ and complexification $G$. All the
  notation will be as in \ref{notazione-varieta}.
\end{say}

\begin{cor} \label{plurr} If $v \in \lieg$ and $v_M(x) =0$ for every
  $x$ outside a set of $\nu$--measure zero, then
  $\exp(\C v) \subset G_\nu$.
\end{cor}
\begin{proof} Since $(i v)_M =J (v_M)$ the result follows immediately
  from the above Lemma. \end{proof}

\begin{prop} \label{P5M} Let $M, G, K$ and $ \mu$ be as in
  \ref{notazione-varieta}.  The function
  \begin{gather}
    \label{defipsim} \PsiM : \proba(M) \times G \ra \R, \quad
    \PsiM(\nu, g) : = \int_M \Psim (x, g) d\nu(x)
  \end{gather}
  is a Kempf-Ness function for $(\pb, G, K)$.  The corresponding function on the symmetric space $X = G/K$ is
  \begin{gather}
    \label{defipsima}
    \psiM _\nu : X \ra \R , \quad \psiM_\nu( gK) := \PsiM (\nu,
    g\meno) = \int_M \psim_x(gK) d\nu(x).
  \end{gather}
The momentum mapping $    \fun : \proba(M) \ra \liek^*$ is given by formula \eqref {bly-intro}, i.e.
  \begin{gather}
 \fun (\nu) : = \int_M \mu(x)
    d\nu(x). \label{def-momento-misure}
  \end{gather}
\end{prop}

\begin{proof}
  Since $\Psim$ is left-invariant with respect to $K$, also $\PsiM$ is
  left-invariant with respect to $K$.  Fix $\nu \in \pb$. By
  differentiation under the integral sign $\PsiM(\nu, \cd)$ is a
  smooth function on $G$ and for $\x \in \liek$ we have
  \begin{gather*}
    \frac{\mathrm{d^2}}{\mathrm{dt}^2 } \PsiM(\nu,\exp(it\x )) =\int_M
    \biggl ( \frac{\mathrm{d^2}}{\mathrm{dt}^2 } \Psim(x,\exp(it\x ))
    \biggr ) d\nu(x) \geq 0,
  \end{gather*}
  since the integrand is non-negative by \ref{P3}.  If
  \begin{gather*}
    \frac{\mathrm{d^2}}{\mathrm{dt}^2 } {\bigg \vert_{t=0} }
    \PsiM(\nu,\exp(it\x ))=0,
  \end{gather*}
  then
  \begin{gather*}
    \frac{\mathrm{d^2}}{\mathrm{dt}^2 } {\bigg \vert_{t=0} }
    \Psim(x,\exp(it\x ))=0 \quad \text{ $\nu$-almost everywhere}.
  \end{gather*}
  Again by \ref{P3} this implies that $\x _M=0$ $\nu$-almost
  everywhere.  By Corollary \ref{plurr} it follows that
  $ \exp(\C \x ) \subset G_\nu$. We have proved
  \ref{P1}--\ref{P3}. The cocycle condition for $\PsiM$ follows
  immediately from the cocycle condition for $\Psim$.  Indeed,
  \begin{gather*}
    \begin{split}
      \PsiM(\nu,gh)&=\int_M \Psim (x, gh) d\nu(x)=\int_M \Psim(x,g) d \nu(x) + \int_M \Psim (gx,h ) d\nu \\
      &=\int_M \Psim(x,g) d \nu(x) + \int_M \Psim (y,h ) d (g\cd \nu) \\
      &=\PsiM (\nu,g)+\PsiM (g\cd \nu, h).
    \end{split}
  \end{gather*}
  Fix $\nu \in \pb$.  It is immediate to verify that the function
  $\psiM$ defined as in \eqref{defpsi} is given by \eqref{defipsima}
  Next we verify that $\psiM_\nu$ is Lipschitz. 
Next we compute the momentum mapping.
\begin{gather*}
  \sx \fun (\nu), v \xs = \int_M \desudtzero \Psi^M(x, \exp(itv)) d\nu
  (x) =
  \int_M \sx \mu(x), v\xs d\nu(x) =\\
  = \sx \int_M \mu(x) d\nu(x), v\xs.
\end{gather*}
This proves that the map $\fun$ defined as in \eqref {momento-astratto} is given by \eqref{def-momento-misure}.
Therefore it is clearly continuous on $\proba(M)$, i.e.  \ref{P6}
  holds.
Finally we verify that $\psiM_\nu$ is a Lipschitz function.
Indeed denoting by $o=K$ the origin in $X$
  \begin{gather*}
    (d\psiM_\nu) _{go} (dL_g(\dga^v(0))) = d\psiM _{g\meno \cd \nu}
    (o) (\dot{\ga}^v(0)) = - \sx \fun(g\meno \cd \nu) , v\xs
    =  \\
    =-\int_M \langle \mu( x),v \rangle d (g\meno\cd \nu) (x).
  \end{gather*}
  Keeping in mind that $M$ is compact and recalling that $L_g$ is an
  isometry of $X$, we get
  $||(d\psiM_\nu)_{go}|| \leq ||\mu||_{L^\infty}$.  Thus $\psiM_{\nu}$
  is a Lipschitz function, i.e. \ref{P5} holds. 
\end{proof}

\begin{teo}
  [Linearization Theorem]
  \label{line}
  Let $M$, $G, K$ and $\mu$ be as in \ref {notazione-varieta}.  If $x$
  is a fixed point of $G$, then there exist an open subset
  $S \subset T_ xM$, stable under the isotropy representation of $G$,
  an open $G$-stable neighbourhood $\Omega$ of $x $ in $M$ and a
  $G$-equivariant biholomorphism $h : S \ra \Omega$.  One can further
  require that $h(0)= x$ and $dh_0 = \Id_{T_xM}$.
\end{teo}
For the proof see \cite[\S 14]{heinzner-schwarz-Cartan}, \cite
{heinzner-huckleberry-Inventiones}, \cite{heinzner-loose} and
\cite{sjamaar-Annals}.

\begin{say}
  \label{notazione-flusso}
  Fix $v\in S(\liek)$.  The \emph{gradient flow} of a function
  $f\in \cinf (M)$ is usually defined as the flow of the vector field
  $-\grad f$.  Let $\{\phi_t\}$ denote the gradient flow of
  $\mu^v$. Since $\grad \mu^v = J v_M = (iv)_M$, we have
  $\phi_t(x) = \exp(-itv) \cd x$.

\end{say}

\begin{prop}\label{limite}
  For any $x\in M$ the limit
  \begin{gather}
\label{def-limes}
    \limes(x) : = \lim_{t\to -\infty} \phi_t(x) = \lim_{t\to +\infty }
    \exp(it\x ) \cd x.
  \end{gather}
  exists.
\end{prop}
\begin{proof}
  Consider the set $L_\alfa(x) $ formed by all points $y\in M$ such
  that there is a sequence $t_n \to -\infty$ with
  $ y = \lim_{n\to \infty} \phi_{t_n}(x) $.  It follows from the
  compactness of $M$ that $L_\alfa(x)$ is a non--empty subset of
  $M$. Moreover $L_\alfa(x)$ is invariant under the flow.  (See
  \cite[Ex. 1 p. 164]{hirsch-differential-topology} for more details.)
  Fix a point $y\in L_\alfa(x)$ and fix a sequence $\{t_n\}$ such that
  $y = \lim_{n\to \infty} \phi_{t_n}(x)$ and $t_n\to -\infty$.  We
  claim that $y$ is a fixed point of the flow, i.e. $v_M(y)=0$.  In
  fact if $v_M(y) \neq 0$, one can linearize the vector field $v_M$ on
  a neighourhood $U$ of $y$. Since $\phi_t$ is a gradient flow one can
  assume that for any point $z \in U$ the flow
  $\phi_t(z) = \exp(-itv)\cd z$ lies out of $U$ for $t $ sufficiently
  negative. But there is $n_0$ such that $\phi_{t_n}(x) \in U$ for
  $n\geq n_0$. So $z = \phi_{t_{n_0}} (x)\in U$, and
  $\phi_{t_n - t_{n_0}} (z) = \phi_{t_n}(x)$ belongs to $U$ for any
  $n$, although $t_n - t_{n_0} \to -\infty$. This yields a
  contradiction and proves that necessarily $v_M(y)=0$ as desired.
  Set $T:=\overline{\exp(\R v)}$.  Then $T$ and its complexification
  $T^\C$ fix $y$.  We get an isotropy action $T^\C \ra \Gl(T_y M)$,
  defined by $a \mapsto da_y$.  The Linearization Theorem \ref{line}
  tells us that one can find an open subset $S \subset T_yM$,
  invariant under the isotropy action, an open $T^\C$--invariant
  subset $\Omega \subset M$ and a $T^\C$--equivariant biholomorphism
  $h : S \ra \Omega$ such that $h(0) = y$, $dh_0 = \Id_{T_yM}$.  Since
  $\exp(\C v) \subset T^\C$ and $\phi_t = \exp(-itv)$, we also have an
  action of $\R$ on $T_yM$, given by $t\mapsto (d\phi_t)_y$.  The
  infinitesimal generator of this action is the symmetric operator
  $H $ corresponding to the Hessian of $-\mu^v$ at $y$, see
  \cite[p. 215]{heinzner-stoetzel-global}.  In other words
  $ (d\phi_t)_y = \exp(tH)$.  Denote by
  $T_yM = V_+ \oplus V_0 \oplus V_-$ the decomposition corresponding
  to the sign of the eigenvalues of $H$.  Since $S$ is invariant by
  $\exp(\R H)$, we have necessarily $V_+ \oplus V_- \subset S$.  Fix a
  small ball $B(0, r) \subset S$.  There is $n_0$ such that
  $h\meno (\phi_{t_n}(x) ) \in B(0,r)$ for any $n\geq n_0$.  Set
  $w:= h\meno (\phi_{t_{n_0}}(x) ) $. Then
  $\phi_{t_n}(x) = \phi_{t_n - t_{n_0}} (\phi_{t_{n_0}}(x))$ and
  $h\meno (\phi_{t_n}(x) ) = \exp((t_n - t_{n_0}) H) \cd w$.  Let
  $w=w_++ w_0+w_-$ be the decomposition with $w_\pm \in V_\pm $ and
  $w_0\in V_0$.  If the component $w_+$ were nonzero, we would have
  $ \lim_{t \to -\infty } || \exp(tH) w|| = +\infty $.  Instead
  $\exp((t_n - t_{n_0}) H) \cd w=h\meno (\phi_{t_n}(x) ) \in B(0, r)$
  for any $n\geq n_0$. Therefore $w_+=0$.  It follows that
  \begin{gather*}
    \lim_{t\to -\infty } \exp(tH) w = w_0,\\
    \lim_{t \to -\infty } \phi_t(x) = \lim_{t \to -\infty }
    \phi_t(\phi_{t_{n_0}}(x)) = h(w_0).
  \end{gather*}
  This proves that the limit exists.
\end{proof}

\begin{say}
  \label{say-frankel}
  Fix again $v\in S( \liek)$ and the notation of
  \ref{notazione-flusso}.  By Frankel Theorem (see
  e.g. \cite[Thm. 2.3, p. 109]{audin-torus-actions} or
  \cite[p. 180]{mcduff-salamon-symplectic}) the function $\mu^v$ is a
  Morse-Bott function with critical points of even index, therefore
  all its local maximum points are global maximum points and all its
  level sets are connected.  Let $c_0 < \cds < c_r$ be the critical
  values of $\mu^v$ and let $C_i:=(\mu^v)\meno ( c_i)$.  Since the
  level sets of $\mu^v$ are connected, the $C_i$'s are exactly the
  connected components of $\Crit(\mu^v)$.  Set
  \begin{gather}
    \label{def-wu}
    W^u_i := \{ x\in M: \limes (x) \in C_i \},
  \end{gather}
  This is the \emph{unstable manifold} of the critical component $C_i$
  for the gradient flow of $\mu^v$.  It follows from the previous
  Proposition that
  \begin{gather}
    \label{scompstabile}
    M = \bigsqcup_{i=0}^r W^u_i.
  \end{gather}
  For any $i$ the map
  \begin{gather*}
    \limes\restr{W^u_i} : W^u_i \ra C_i
  \end{gather*}
  is a smooth fibration with fibres diffeomorphic to $\R^{l_i}$ where
  ${l_i}$ is the index (of negativity) of the critical submanifold
  $C_i$. Since all local maximum points of $\mu^v$ are global maximum
  points, we have $\dim_\R C_i + l_i =\dim_\R M$ if and only if
  $i=r$. This means that $W_i^u$ is open only for $i=r$. It follows
  from \eqref{scompstabile} that $W^u_r $ is also dense.
\end{say}

\begin{teo} \label{calcolo-lambda} With the notation above we have
  \begin{gather*}
    \la_\nu(\e(-v)) = \sum_{i=0}^r c_i \cd \nu(W^u_i).
  \end{gather*}
\end{teo}
\begin{proof}
  Using the definition of $\la_\nu$, \eqref{defipsim} and
  differentiation under the integral we get
  \begin{gather*}
    \la_\nu (\e (-v)) = \lim_{t\to +\infty} \desudt \PsiM (\nu,
    \exp(itv) ) =\\
    =\lim_{t\to +\infty} \int_M \Bigl ( \desudt \Psim (x, \exp(itv))
    \Bigr ) d\nu(x) .
  \end{gather*}
  By \eqref{derivata-t}
  \begin{gather*}
    \desudt \bigg \vert _{t=t_o} \Psim (x, \exp(it\x ) ) = \mu^\x
    (\exp(it_0 \x ) \cd x).
  \end{gather*}
  Since $\mu^v$ is a bounded we can apply the dominated convergence
  theorem:
  \begin{gather*}
    \la_\nu (\e (-v)) = \lim_{t\to +\infty} \int_M \mu^v (\exp
    (itv)\cd x) d\nu(x) =\\
    = \int_M \Bigl ( \lim_{t\to +\infty} \mu^v (\exp
    (itv)\cd x) \Bigr ) d\nu(x) = \int_M \mu^v (\alfa(x)) d\nu(x) =\\
    = \sum_{i=0}^r \int_{W_i^u} \mu^v (\alfa(x)) d\nu(x) .
  \end{gather*}
  For $x\in W_i^u$, $\alfa(x) \in C_i$, so $\mu^v(\alfa(x)) =
  c_i$. Thus
  \begin{gather*}
    \int_{W_i^u} \mu^v (\alfa(x)) d\nu(x) = c_i \cd \nu(W_i^u).
  \end{gather*}
  This proves the theorem.
\end{proof}

\begin{say}
Let $\convo$ denote the convex hull of $\mu(M) \subset \liek^*$ and
let $\intec $ denote the interior of $\convo$ as a subset of
$\liek^*$.
Remark that  the action of $K$ on $M$
is \emph{almost effective} (i.e. if $v \in \liek$ and $v_M=0$, then
$v=0$) if and only if $\mu(M)$ is full in $\liek^*$ (i.e. it is not contained in any affine hyperplane) if and only if   $\intec$ is non-empty.
\end{say}

\begin{lemma}\label{lemma-convesso}
  If $\nu$ is a probability measure on $M$, then
  $\fun( \nu) \in \convo$.
\end{lemma}
\begin{proof}
  If $\nu\in \pb$, then $\fun(\nu) $ is the center of gravity of the
  measure $\mu_* \nu \in \proba (\mu(M))$. Therefore $\fun(\nu) $ lies
  in the convex hull of $\mu(M)$.
\end{proof}

\begin{say}\label{smooth-measure}
  On a differentiable manifold there is no preferred measure, but
  there is a well defined class of measures: those that in any chart
  have a smooth strictly positive density with respect to the Lebesgue
  measure of the chart.  These are called \emph{smooth strictly
    positive measures} on $M$.  Any two such measures are absolutely
  continuous with respect to one another.
\end{say}

 \begin{defin}\label{def-ac}
   Let $\ac(M)$ denote the set of the probability measures on $M$ that
   are absolutely continuous with respect to one smooth strictly
   positive measure (and hence with respect to any such measure).
 \end{defin}

\begin{defin}
  Let $\W (M, K, \om, \mu)$ (or $\W(M,K)$ for brevity) denote the set
  of probability measures on $M$ that satisfy the following condition:
  for every $v\in S(\liek)$, the open unstable manifold has full
  measure. In the notation of \ref{say-frankel} this means that
  $\nu (W^u_r ) = 1$.
\end{defin}

\begin{say}
  In words $\nu \in \W(M,K)$ if $\nu$ is concentrated on the open
  unstable manifold. Since the other unstable manifolds are
  submanifolds of positive codimension in $M$, it is clear that
  $\ac(M) \subset \W(M,K)$.
\end{say}

An easy consequence of Theorem \ref{calcolo-lambda} is the following
result.
\begin{cor}
  \label{sup}
  If $\mis \in \W(M,K)$, then for any $v\in S( \liek)$
  \begin{gather*}
    \la_\nu(\e (-v))= \max_{x\in M} \mu^v(x).
  \end{gather*}
\end{cor}
\begin{proof}
  By Theorem \ref{calcolo-lambda},
  $\la_\nu(\e (-v))=\sum_{i=0}^r c_i \nu (W^v_i)$. Since
  $\nu(W^v_r)=1$, $\nu(W^v_i)=0$ for $i=0,\ldots,r-1$, so
  $\la_\nu(\e (-v))= c_r=\max_{x\in M} \mu^v(x).$
\end{proof}

\begin{teo}
  \label{W-stabile}
  If $\nu \in \W(M,K)$ and $0\in \intec $, then $\nu$ is stable. In
  particular it is polystable, so there is $g\in G$ such that
  $\fun(g\pf\nu)=0$.
\end{teo}
\begin{proof}
  Since $0 \in \intec $, for each $v\neq 0$ the function $\mu^v$
  attains both positive and negative values. So the maximum of $\mu^v$
  is positive. By the previous corollary $\la_\nu(\e(-v)) >0$. The
  result follows applying Theorem \ref {stabile}.
\end{proof}

\begin{cor}\label{stable-dense}
  If $0\in \intec $, then the set
  $\proba_s(M) :=\{\nu\in \proba (M):\, \nu\ \mathrm{is\ stable} \}$
  is open and dense in $\pb$.
\end{cor}
\begin{proof}
  By Corollary \ref{stable-open-abstract-setting} the set
  $\proba_s(M)$ is open. By Theorem \ref{W-stabile} any smooth measure
  is stable. Hence it is enough to prove that smooth measures are
  dense. It is easy to check that for any Dirac measure $\delta_y$
  there exists a sequence of smooth measures $\nu_n$ such that
  $\nu_n \weak \delta_y$. Hence any convex combination of Dirac
  measures is the weak limit of a sequence of smooth measures. Since
  $\proba(M)$ is a compact convex set in $\misu(M)$ (endowed with the
  weak topology) and its extremal points are exactly the Dirac
  measures, the Krein-Milman Theorem \cite[p. 440]{dunford-schwartz-1}
  implies that convex combinations of Dirac measures are
  dense. Therefore also smooth probability measures are dense and so
  $\proba_s(M)$ is dense in $\pb$.  (See Lemma 3.6 in \cite[p. 316]
  {kapovich-leeb-millson-convex-JDG} for a similar argument.)
\end{proof}

\begin{say}
  \label{puzzola}
  We point out that for an almost effective action, up to shifting the
  momentum mapping $\mu$, the condition $0\in \intec $ is always
  satisfied. This is the content of the following lemma.  It implies
  that $0 \in \into$ when $K$ is semisimple and the action is almost
  effective.  In the following we fix an $\Ad$--invariant scalar
  product on $\liek$ and we think of the momentum mapping as a
  $\liek$-valued map.
\end{say}

\begin{lemma}
  Let $(M, \om)$, $K$ and $\mu$ be as in \ref{notazione-varieta}.
  Assume that the action is almost effective.  Then
  $ \intec \cap \liez (\liek) \neq \vacuo$.
\end{lemma}
\begin{proof}
  Fix a maximal torus $T\subset K$ and let $\liet$ be its Lie algebra
  and $\pi: \liek \ra \liet$ the orthogonal projection. Then
  $P:=\pi (\mu(M)) $ is the momentum polytope for the $T$-action. Let
  $\{v_1 , \lds, v_q\}$ be the set of vertices of $P$. Then
  $b:=(v_1 + \cds + v_q) /q$ is the centroid of $P$
  \cite[p. 60]{berger-geometry-1}.  We claim that $b \in \inte P$.  In
  fact, if $b$ were a boundary point, there would be a a proper face
  containing $b$. Since any face of a polytope is exposed
  \cite{schneider-convex-bodies}, there would exist a nonzero vector
  $u\in \liek$ such that $\sx b,u \xs = c$, where
  $c:=\max_{x\in P} \sx x, u\xs$.  So $b$ would belong to the face
  $F_u(P)=\{y \in P:\,\langle x, u \rangle=c \}$. After reordering the
  vertices we can assume that $\sx v_j , u \xs = c $ if
  $1\leq j \leq p$ and $\sx v_j , u \xs < c$ if $p < j \leq q$.  This
  means that $v_1 , \lds, v_p $ are the extremal points of
  $F_u(P)$. So there would exist $\la_i\in [0,1]$ such that
  $b = \la_1 v_1 + \cds + \la_p v_p$.  Therefore
  \begin{gather*}
    \sum_{i=1}^ p \bigl ( \la_i - \frac{1}{q} \bigr ) v_i = \sum_{i>p}
    \frac{1}{q} v_i.
  \end{gather*}
  But the left hand side lies in the hyperplane
  $ \{ x\in \liet: \sx x, u\xs =c \}$, while the right hand side lies
  in the half-space $ \{ x\in \liet: \sx x, u\xs <c \}$. Thus we have
  $p=q$ and $P=F_u(P)$. But this is absurd since $F_u (P)$ is a proper
  face. Thus we have proved that $b \in \inte P$.  By the main Theorem
  in \cite{bgh-israel-p} the faces of $\convo$ correspond to the faces
  of $P$. Hence the interior of $\convo$ corresponds to the interior
  of $P$, that is $K\cd \inte P = \intec $. So $b\in \intec$.  Finally
  $\mu(M)$ is $K$-invariant, so $P=\pi(\mu(M))$ is invariant for the
  Weyl group $\Weyl = \Weyl(K,T)$. Therefore also $b$ is fixed by
  $\Weyl$.  This proves that $b$ lies in $\liez(\liek)$.
\end{proof}

    \begin{cor}
      Let $(M, \om)$ be a compact K\"ahler manifold and let $K$ be a
      compact group acting on $M$ almost effectively and in
      Hamiltonian fashion with momentum mapping $\mu: M \ra \liek^*$.
      If $\nu \in \ac(M) $, then $G_\nu$ is compact.  If
      $\nu \in \ac(M) $ is $K$--invariant, then $G_\nu=K$.
    \end{cor}

    \begin{proof}
      By the above lemma, up to shifting $\mu$ by an element of the
      center we can assume that $0\in \intec$.  By Theorem
      \ref{W-stabile} $\nu$ is stable.  By Corollary \ref{stabcomp}
      $G_\nu$ is compact. If $\nu$ is $K$--invariant, then
      $K\subset G_\nu$. Since $K$ is a maximal compact subgroup of
      $G$, we get $G_\nu=K$.
    \end{proof}

    \section{The construction of Hersch and Bourguignon-Li-Yau}
    \label{sec:bly}

\begin{say}
  In various situations it is interesting to know how the map $\fun$
  behaves on the orbit $G\cd \nu$, where $\nu \in \pb$.  Set
  \begin{gather}
    \label{def-bly}
    \bly_\nu : G \lra \liek^*, \quad \bly_\nu(a): = \fun (a\cd \nu).
  \end{gather}
  This map was used for the first time by Hersch \cite{hersch}, in the
  case $M=S^2$, to get upper bounds for $\la_1$. For the same purpose
  it was generalized by Bourguignon, Li and Yau
  \cite{bourguignon-li-yau} to the case $M=\mathbb{P}^n(\C)$. We
  further generalized it to arbitrary flag manifolds in
  \cite{biliotti-ghigi-American}.  In this section we prove rather
  general theorems that extend the results in these papers to actions
  on arbitrary K\"ahler manifolds.

  Recall from \ref{notazione-varieta} that $\mu$ is the momentum
  mapping with respect to the symplectic form $\om$ and that $g$ is
  the K\"ahler metric corresponding to $\om$.  For $X,Y\in \XX(M)$ set
  \begin{gather*}
    (X,Y)_{L^2(g, \nu)} := \int_M g(X(x), Y(x)) d\nu(x).
  \end{gather*}
  In all this section we assume that the action of $K$ is almost
  effective.
\end{say}

    \begin{lemma}
      If $v, w\in \liek$ and $\nu \in \proba(M)$, then
      \begin{gather*}
        \desudtzero \sx \fun( \exp(itv )\cd \nu) , w \xs = ( v_M,
        w_M)_{L^2(g,\nu)}.
      \end{gather*}
    \end{lemma}
    \begin{proof}
      \begin{gather*}
        \desudtzero \sx \fun( \exp(itv) \cd \nu) , w \xs =
        \desudtzero \sx \int_M \mu(y)\,  d(\exp(itv) \cd \nu)(y) , w \xs =\\
        = \desudtzero \int_M \mu^w(y)\, d(\exp(itv) \cd \nu)(y)
        = \desudtzero \int_M \mu^w(\exp(itv)\cd x) \, d\nu(x)= \\
        = \int_M i_{w_M} \om (Jv_M) \, d\nu =\int_M g ( w_M, v_M) d\nu
        = (v_M, w_M)_{L^2(g,\nu)}.
      \end{gather*}
    \end{proof}

    \begin{teo}
      \label{maxirank}
      If $\nu \in \proba(M)$ and $G_\nu$ is compact, then the map
      $F_\nu: G \ra \liek^*$ defined in \eqref{def-bly} is a smooth
      submersion and its image is contained in $ \intec $.
    \end{teo}

\begin{proof}
  If $a\in G$ and $v \in \liek$, consider the curve
  $\alf(t):= \exp(itv) a$. Set $\tnu:= a\cd \nu$. Then for any
  $w\in \liek$
  \begin{gather*}
    \sx d F_\nu(\dalf(0)), w\xs = \desudtzero \sx \fun (\exp(itv)\cd
    \tnu ) , w \xs = (v_M, w_M) _{L^2(g, \tnu)}.
  \end{gather*}
  If $dF_v(\dalf(0)) =0$, choose $w=v$. Then
  $||v_M ||_{L^2(g, \tnu)}=0$, i.e. $v_M=0$ $\tnu$--a.e.
  By Corollary \ref{plurr} $\exp(\C v) \subset G_{\tnu}$.  Since
  $G_{\tnu}$ is compact, $v =0$.  This proves that $dF_\nu$ is
  injective on the subspace $dR_a(e) (i\liek) \subset T_aG$. By
  dimension reasons $T_aG = \ker dF_\nu \oplus dR_a(e) (i\liek)$ and
  $dF_\nu$ is onto. Therefore $F_\nu(G)$ is an open subset of
  $\liek^*$. Since it is contained in $\convo$ we have
  $F_\nu(G) \subset \intec$.
\end{proof}
\begin{say}
  The first assertion of the previous theorem is analogous to a fact
  that is well-known in the classical theory of Hamiltonian actions on
  a K\"ahler manifold: if the stabilizer of a point is compact, the
  restriction of the momentum mapping to its orbit is a submersion.
  See e.g. \cite[Prop. 6.1]{heinzner-schwarz-Cartan}.
\end{say}
\begin{lemma}
  If $\mis \in \proba(M)$ and $G_\nu$ is compact, then the function
  $\psim_\nu$, defined in \eqref{defipsima}, is strictly convex.
\end{lemma}
\begin{proof}
  Convexity is proven in Lemma \ref{convstab}. We need to check strict
  convexity.  Using the cocycle condition we can restrict to geodesics
  passing through $o=K \in X$.  Let $\alfa(t) = \exp(itv)K $ be such a
  geodesic.  If
  \begin{gather*}
    0= \frac{\mathrm{ d} ^2 } {\mathrm {dt}^2} \bigg |_{t=t_0}
    \psiM_\nu \circ \alfa (t)= \int_M \frac{\mathrm{d}^2}{\mathrm
      {dt}^2} \psim_x(\alfa(t)) d\mis(x),
  \end{gather*}
  then
  \begin{gather*}
    \frac{\mathrm{d}^2}{\mathrm {dt}^2} \psim_x(\alfa(t)) = 0 \quad
    \mis-\text{a.e. in }M.
  \end{gather*}
  By \ref{P3} this implies that $v_M = 0$ $\mis$-a.e., so
  $\exp(\C v) \subset G_\nu$. Since $G_\nu$ is compact, $v=0$. This
  proves that $\psiM_\mis$ is strictly convex along geodesics of $X$
  that pass through the origin $o=K$. The usual argument with the
  cocycle condition yields strict convexity along any geodesic of $X$.
\end{proof}

\begin{say}
  Assume now that $K=T$ is a compact torus, so $G=T^\C$.  Set
  \begin{gather*}
    \tf_\nu: \liet \ra \liet^*, \qquad \tf_\nu(v) : = F_\nu(\exp(iv)).
  \end{gather*}
  Denote by $P$ the momentum polytope i.e. the image of
  $\mu : M \ra \liet^*$. By the Atiyah-Guillemin-Sternberg convexity
  theorem \cite{atiyah-commuting,guillemin-sternberg-convexity-1} $P$
  is a polytope.
\end{say}

\begin{prop} If $\nu \in \W(M, T, \om, \mu)$, then $\tf_\nu$ is a
  diffeomorphism of $\liet$ onto the interior of the momentum
  polytope.
\end{prop}
\begin{proof}
  Let $\pi : \liet \ra X:=T^\C/T$ be the diffeomorphism
  $\pi (v):= \exp(-iv)T $.  Since $T$ is abelian, the geodesics of $X$
  are images under the map $\pi$ of affine lines in $\liet$.  Since
  $\psiM_\nu$ is strictly convex on $X$, the function
  $f := \psiM_\nu \circ \pi: \liet \ra \R $, is strictly convex on
  $\liet$.  Note that
  \begin{gather*}
    f(v) =\int_M \psim_x (\exp(-iv)T)d\nu(x) .
  \end{gather*}
  Using the commutativity of $T$, the cocycle condition
  \eqref{cociclo-psi} and \eqref{caso-mundet} we have
  \begin{gather*}
    \desudtzero{} \kn_x(\exp (-i(w+tv))) = \desudtzero{} \kn_x(\exp
    (-iw) \exp(-itv)T)  \\
    = \desudtzero{} \biggl ( \kn_{\exp(iw)x}(\exp (-itv)T) +
    \kn_x(\exp(-iw)T) \biggr)= \sx \mu (\exp(iw)\cd x), v\xs
  \end{gather*}
  \begin{gather*}
    df(w) v = \int_M \desudtzero \kn_x(\exp (-i(w+tv))T) d\mu(x) = \\
    = \sx F_\mis(\exp(iw)), v\xs = \sx \tf_\mis (w), v\xs.
  \end{gather*}
  Hence $df(w) = \tf_\mis(w)$.  Since $f$ is strictly convex, a basic
  result in convex analysis \cite[p. 122]{guglielmino-torico} ensures
  that $df : \liet \ra \liet^*, x\mapsto df(x)$ is a diffeomorphism
  onto an open convex subset $U$ of $ \liet^*$.  Therefore
  $ \tf_\nu (\liet) = df(\liet) = U$ is an open convex subset of
  $\liet^*$ and $\tf_\nu : \liet \ra U$ is a diffeomorphism.  By Lemma
  \ref{lemma-convesso} $U\subset P$.  Since $U$ is an open subset of
  $\liet$ we have $ U \subset \inte P$.  We need to show that
  $U = \inte P$.  Assume by contradiction that $U \subsetneq \inte P$.
  Since both $U$ and $\inte P$ are convex, we get
  $\overline{U} \subsetneq P$. Fix $x_0 \in P \setminus \overline{U}$
  and $x_1 \in U$. Set $x_t : = (1-t) x_0 +t x_1$ and
  $\tau : = \inf \{t\in [0,1]: x_t \in \overline{ U}\}$.  Since
  $\overline{U}$ is closed $x_\tau\in \overline{U}$ and
  $\tau \in(0,1)$.  Moreover $x_\tau \in \partial U$. Since
  $x_1 \in U \subset \inte P$ and $\tau >0$, it follows that
  $x_\tau \in \inte P$. So $y:=x_\tau \in \partial U \cap \inte P$.
  Any boundary point of a compact convex set lies in some exposed
  face, i.e.  it admits a support hyperplane
  \cite{schneider-convex-bodies}. So there is $v\in \liet$, $v\neq 0$,
  such that
  \begin{gather*}
    \sx y, v \xs = \max_{\overline{U} } \sx \cd , v \xs = \sup_{U }
    \sx \cd , v \xs = \sup _{w\in \liet} \sx \tf_\nu (w), v\xs .
  \end{gather*}
  Since $\nu \in \W(M, T, \om, \mu)$, Corollary \ref{sup} yields that
  \begin{gather*}
    \la_\nu(\e (-v))= \max_{x\in M} \mu^v(x) .
  \end{gather*}
  Moreover
  \begin{gather*}
    \la_\nu (\e (-v)) = \lim_{t\to +\infty} \int_M \mu^v (\exp
    (itv)\cd x) d\nu(x) = \lim_{t\to +\infty} \sx \tf_\nu(\exp(itv),
    v\xs.
  \end{gather*}
  (See the proof of Theorem \ref{calcolo-lambda}.)  Therefore
  \begin{gather*}
    \sup_{w\in \liet} \sx \tf_\nu(w), v\xs \geq \la_\nu(\e(-v)) .
  \end{gather*}
  Summing up
  \begin{gather*}
    \sx y, v\xs = \sup_{w\in \liet} \sx \tf_\nu(w), v\xs \geq
    \la_\nu(\e(-v)) = \max_{x\in M} \mu^v(x) = \max_P \sx \cd , v\xs.
  \end{gather*}
  This means that the linear function $\sx \cd, v\xs$ attains its
  maximum on $P$ at the point $ y \in \inte P$.  Since $P$ is a convex
  set, this implies $v=0$, a contradiction.  Therefore we have indeed
  $\overline{U} = P$ and $U = \inte P$.
\end{proof}

\begin{teo}\label{T-propria}
  If $\nu \in \W(M, T, \om, \mu)$ and the action of $T$ on $M$ is
  almost effective, then $F_\nu : T^\C \ra \inte P$ is a surjective
  submersion with compact fibres.
\end{teo}
\begin{proof}
  Since $T^\C$ is abelian, the map $F_\nu$ is $T$-invariant: if
  $k\in T$, then $F_\nu( kg) = \Ad (k) F_\nu(g) = F_\nu(g)$.  Let
  $\phi : T\times \liet \ra T^\C$ be the diffeomorfism
  $\phi(k, v) : = k \cd \exp(iv)$ and let
  $\operatorname{pr}_2: T\times \liet \ra \liet$ be the projection on
  the second factor. Then
  $F_\nu = \tf_\nu \circ \operatorname{pr}_2 \circ \phi\meno$.
  Therefore $F_\nu$ is a proper submersion onto $\inte P$ and
  its the fibres are the $T$--cosets.
\end{proof}
\begin{say}
  Consider the following example: $M=S^2$ with $T=S^1$ acting by
  rotation around the $z$-axis; $\nu$ is the measure concentrated at
  the South pole.  Then the group $T^\C = \C^*$ acts as complex
  dilations of the Riemann sphere leaving the South pole fixed. Thus
  the map $F_\nu$ is constant. In particular it is not a submersion
  from $\C^*$ to $\liet \cong \R$ . Clearly $\nu $ does not belong to
  $\W(S^2, T)$. So in some sense the previous result is sharp.  We are
  going to prove a similar result in the non-abelian case. This will
  be very general, though not as sharp as the abelian one. Indeed we
  will need a technical condition that is dealt with in the next
  lemma.
\end{say}
\begin{lemma}
  Let $M$ be a compact manifold and let $K$ be a compact Lie group
  acting continuously on $M$.  Let $\nu_0\in \pb$ be $K$-invariant and
  let $\nu \in \pb$ be absolutely continuous with respect to
  $\nu_0$. Let $k_n$ be a sequence in $K$ converging to $ k$.  Then
  $k_n \cd \nu \to k \cd \nu$ in the norm \eqref{totvar}.
\end{lemma}
\begin{proof}
  By the Radon-Nikodym theorem there is a non-negative function
  $\phi\in L^1(M, \nu_0)$ such that $d \nu= \phi \cd d\nu_0$.  We
  claim that
  \begin{gather*} ||k_n \cd \nu - k \cd \nu|| \leq || \phi\circ
    k_n\meno - \phi\circ k||_{L^1( \nu_0)}.
  \end{gather*}
  If $f$ is a bounded measurable function on $M$, then
  \begin{gather*}
    \int_M f\, d(k\cd \nu) = \int_M f(k x) \phi(x) d\nu_0(x) = \int_M
    f(y) \phi(k\meno\cd y) d\nu_0(y),
  \end{gather*}
  and similarly for $k_n$. So
  \begin{gather*}
    ||k_n \cd \nu - k \cd \nu|| = \\
    =\sup \Bigl \{
    \int_M f d (k_n \cd \nu) -  \int_M f d (k \cd \nu) : f\in C(M), \sup_M |f| \leq 1\Bigr \} = \\
    = \sup \Bigl \{ \int_M f ( \phi \circ k_n\meno - \phi\circ k\meno
    ) d \nu_0 : f\in C(M), \sup_M |f| \leq 1
    \Bigr \} \leq \\
    \leq || \phi \circ k_n\meno - \phi\circ k \meno||_{L^1( \nu_0)}.
  \end{gather*}
  This proves the claim.  Given $\eps>0$ fix a continuous function
  $\phi_0$ such that
  \begin{gather*}
    ||\phi - \phi_0 ||_{L^1( \nu_0)} < \eps.
  \end{gather*}
  Using the $K$--invariance of $\nu_0$ we get
  \begin{gather*}
    ||\phi \circ k_n \meno - \phi \circ k\meno||_{L^1( \nu_0)} \leq\\
    \leq ||\phi \circ k_n \meno - \phi_0 \circ k_n\meno||_{L^1(
      \nu_0)} + ||\phi_0 \circ k_n \meno - \phi_0 \circ
    k\meno||_{L^1( \nu_0)} + \\
    +||\phi_0 \circ k\meno - \phi \circ k\meno ||_{L^1( \nu_0)} =\\
    = 2 ||\phi -\phi_0||_{L^1( \nu_0)} + ||\phi_0 \circ k_n\meno -
    \phi_0 \circ k\meno||_{L^1( \nu_0)}
    <\\
    < 2\eps + ||\phi_0 \circ k_n\meno - \phi_0 \circ k\meno||_{L^1(
      \nu_0)}.
  \end{gather*}
  As $\phi_0$ is continuous there is $\delta >0$ such that
  $ |\phi_0(x) - \phi_0 (y) | < \eps$ if $d(x, y) < \delta $.  The
  action of $K$ on $M$ being continuous and $M$ being compact imply
  that $k_n \to k $ uniformly on $M$ (see the proof of Lemma
  \ref{prob-continua}). Thus there is $n_0 $ such that for any
  $n\geq n_0$ and for any $x\in M$, $d(k_n\cd x, k \cd x) < \delta$.
  Therefore for $n\geq n_0$
  \begin{gather*}
    ||k_n \cd \nu - k \cd \nu|| \leq ||\phi \circ k_n \meno - \phi
    \circ k\meno||_{L^1( \nu_0)} < (2 + \nu_0(M) ) \eps = 3\eps.
  \end{gather*}
  This proves the lemma.
\end{proof}

\begin{prop}\label{propina}
  Let $(M, \om)$ be a \Keler manifold and let $K$ be a compact group
  acting isometrically and almost effectively on $M$ with momentum
  mapping $\mu: M \ra \liek^*$.  If $\nu_0\in \pb$ is $K$-invariant
  and $\nu \in \W(M,K)$ is absolutely continuous with respect to
  $\nu_0$, then $F_\nu(G) = \into$ and $F_\nu : G\ra \into$ is a
  fibration with compact connected fibres.
\end{prop}
\begin{proof}
  Since $\nu \in \W(M,K)$, we know that after appropriately shifting
  $\mu$, the measure $\nu$ is stable (Theorem \ref{W-stabile}, hence
  $G_\nu$ is compact (Corollary \ref {stabcomp}).  By Theorem
  \ref{maxirank} $F_\nu$ is of maximal rank. Therefore its image is an
  open subset of $\liek$. Since it is contained in $\convo$, it is in
  fact contained in the interior of $\convo$, i.e. in $\intec$.  We
  claim that if $F_\nu$ is regarded as a map $F_\nu: G \ra \into$,
  then it is proper.  Indeed let $\{g_n \}\subset G$ be a sequence
  such that $\{F_\nu(g_n)\}$ converges to a point of $\into$. We have
  to show that some subsequence of $\{g_n\}$ converges.  Let
  $T\subset G$ be a maximal torus. Since $G=KT^\C K$, we write
  $g_n = q_n \exp(iv_n) k_n\meno$ with $k_n, q_n \in K$ and
  $v_n \in \liet$. Passing to subsequences we may assume that
  $k_n \to k$ and $q_n \to q$.  From the fact that
  $F_\nu (kg) = \Ad(k) F_\nu(g)$, we immediately deduce that the
  sequence $\{F_\nu(\exp(iv_n) k_n\meno)\}$ is also convergent in
  $\into$.  We claim that
  \begin{gather} \label{sss} F_\nu(\exp(iv_n)k_n \meno ) -
    F_\nu(\exp(iv_n)k \meno ) \to 0.
  \end{gather}
  Fix $w\in \liek$. Then
  \begin{align*}
    \sx F_\nu (\exp  (&iv_n) k_n\meno) , w\xs =
                      \int_M \mu^w (\exp(iv_n)k_n\meno \cd x) d\nu(x) = \\
                    &= \int_M \mu^w ( \exp(iv_n) y) d (k_n\meno \cd\nu)(y) .
    \\
    \sx F_\nu(\exp (&iv_n)k_n \meno ) -
                     F_\nu(\exp(iv_n)k \meno), w\xs = \\
                    & = \int_M \mu^w \circ \exp(iv_n) \, d (k_n\meno \cd\nu) - \int_M
                      \mu^w \circ \exp(iv_n) \, d (k\meno \cd\nu) .
    \\
    \Big |\sx F_\nu( \exp &(iv_n)k_n \meno - F_\nu(\exp(iv_n)k \meno , w\xs
                            \Big | \leq
    \\
                    &  \leq || \mu^w \circ \exp(iv_n)||_{C(M)} \cd || k_n\meno \cd\nu -
                      k\meno \cd\nu ||\leq \\
                    & \leq \sup_M | \mu| \cd || k_n\meno \cd\nu - k\meno \cd\nu ||.
    \end{align*}
  Therefore
  \begin{gather*}
    \Big | F_\nu(\exp(iv_n)k_n \meno - F_\nu(\exp(iv_n)k \meno \Big |
    \leq \sup_M | \mu| \cd || k_n\meno \cd\nu - k\meno \cd\nu ||.
  \end{gather*}
  Since $\nu_0$ is $K$--invariant and $\nu$ is absolutely continuous
  with respect to $\nu_0$, the previous lemma ensures that
  $ || k_n\meno \cd\nu - k\meno \cd\nu || \to 0$.  Thus \eqref{sss} is
  proved.  Since $ \{ F_\nu(\exp(iv_n)k_n \meno ) \}$ is convergent in
  $\into$, it follows that also $\{ F_\nu(\exp(iv_n)k \meno ) \} $
  converges to some point of $\into$. And the same is true for
  $F_\nu(k \exp(iv_n)k\meno)$ $ = $ $ \Ad(k) F_\nu(\exp(iv_n)k) $.
  The points $k \exp(iv_n) k\meno$ belong to the maximal torus
  $S:=k T k\meno$.  Let $\pi: \liek^* \ra \lies^*$ denote the
  restriction.  Then $\mu_S:=\pi \circ \mu : M\ra \lies^*$ is a
  momentum mapping for the action of $S$ on $M$.  Let $P$ denote the
  momentum polytope.  Then $P=\pi (\convo) $.  Note also that
  $\pi(\into) \subset P^0$ since $\pi$ is obviously an open map.  By
  Theorem \ref{T-propria} the map
  \begin{gather*}
    F_\nu^S : S^\C \lra \lies^*, \quad F_\nu^S (a) : = \int_M
    \mu_S(a\cd x) d\nu(x),
  \end{gather*}
  is a proper submersion onto $\inte P$. (Note that
  $\W(M,K) \subset \W(M, T)$.)  But
  $F_\nu^S = \pi \circ F_\nu \restr{S^\C}$, that is the following
  diagram commutes:
  \begin{equation*}
\begin{tikzcd}[column sep=large]
  S^\C \arrow[hook]{r} \arrow {rrd}[below]  {F_\nu^S}
  &  G \arrow {r}{F_\nu} & \liek^* \arrow {d}{r}{\pi}\\
  & & \lies^*.
  \end{tikzcd}
\end{equation*}
  Since $\{k\exp(iv_n)k\meno\} \subset S^\C$ and
  $F^S_\nu(k \exp(iv_n)k\meno)$ converges to some point of $P^0$, 
  we conclude that the sequences $\{k\exp(iv_n)k\meno\} $ and
  $\{\exp(iv_n)\}$ admit convergent subsequences.  This proves that
  $F_\nu : G \ra \into$ is proper. Hence it is a closed map. As it is
  also open, it is onto. Moreover it is a locally trivial fibration by
  Ehresmann theorem. Since the base is contractible, the fibration is
  trivial, i.e. $G$ is diffeomorphic to a product $\Omega \times F$
  where $F$ is the fibre. Since $G$ is connected it follows that $F$
  is connected.
\end{proof}
\begin{teo}\label{propria}
  Let $(M,g, \om)$ be a \Keler manifold and let $K$ be a compact group
  acting isometrically and almost effectively on $M$ with momentum
  mapping $\mu: M \ra \liek^*$.  If $\nu \in \ac(M)$, then
  $F_\nu(G) = \into$ and $F_\nu : G\ra \into$ is a fibration with
  compact connected fibres.
\end{teo}
\begin{proof}
  Fix a $K$-invariant Riemannian metric on $M$ and denote by $\nu_0$
  the the normalized Riemannian measure. Then $\nu$ is absolutely
  continuous with respect to $\nu_0$ and $\nu \in \W(M,K)$. Thus the
  theorem follows from the previous proposition.
\end{proof}

\section{Applications}
\label{sec:applications}
\begin{say}
  In this section we describe some applications of the results
  obtained in the previous sections.  In \ref
  {autospazi}--\ref{flag-stability} we give an explicit
  characterization of stable, semi-stable and polystable measures on
  $\PP^n$. These have attracted interest in various contexts, e.g. in
  the study of balanced metrics \cite{donaldson-numerical-results} and
  in the study of the loop space of $S^2$ \cite{millson-zombro}.

  In the rest of the section we make some remarks on the relation of
  our results to the problem of upper bounds for $\la_1$ in the style
  of Bourguignon-Li-Yau \cite{bourguignon-li-yau}.

\end{say}

\begin{say}
  \label{autospazi}
  We wish to characterize stable, semi-stable and polystable measures
  on $\PP^n$ endowed with the Fubini-Study metric and the standard
  action of $\Sl(n+1, \C)$.  Let $v \in \su(n+1)$, $v\neq 0$, let
  $c_0 < \cds < c_r$ be the eigenvalues of $iv$ and let
  $V:=\C^{n+1} = V_{0} \oplus \cds \oplus V_{r}$ be the eigenspace
  decomposition of $v$, so that $v$ acts on $V_i$ as the
  multiplication by $-ic_i$.  Set
  \begin{gather}
    \label{eq:4}
    E_i : = \bigoplus_{j\leq i} V_j, \quad L_i:=\PP(E_i), \quad L_{-1}
    :=\vacuo.
  \end{gather}
\end{say}
\begin{lemma}
  \label{proiettivo}
  The critical values of $\mu^v$ are $c_0 < \cds <c_r$.  The critical
  component corresponding to $c_i$ is $C_i = \PP(V_i)$ and the
  unstable manifold of $C_i$ is $W_i^u = L_i - L_{i-1}$.
\end{lemma}
\begin{proof}
  Fix a point $x=[z] \in \PP^n(\C)$ and decompose $z$ according to the
  eigenspace decomposition: $z= z_0 + \cds + z_r$.  Then
  \begin{gather}
    \label{eq:6}
    \mu^v (x) = \frac{i \sx v (z), z\xs }{|z|^2} = \sum_{i=0}^r c_i
    \frac{|z_i|^2}{|z|^2}
  \end{gather}
  Hence $C_i = (\mu^v)\meno (c_i) = \PP(V_i)$.  Next let $i$ be such
  that $z_i \neq 0$ and $z_j = 0$ for $j>i$.  Then
  \begin{gather*}
    \exp(itv) \cd x = [ e^{c_0t} z_0 + \cds + e^{c_i t} z_i] = \\
    = [ e^{(c_0 - c_i)t } z_1 + \cds + e^{(c_{i-1} - c_i)t}
    z_{i-1} + z_i] ,\\
    \lim_{t\to +\infty} \exp(i tv) \cd x = [z_i].
  \end{gather*}
  It follows that $ \{[z ] \in \PP^n : z=z_0 + \cds +z_i $ with
  $z_i \neq 0\} = L_i-L_{i-1} \subset W^u_i $. Since both $\{W^u_i \}$
  and $\{L_i - L_{i-1}\}$ are partitions of $\PP^n$ we deduce
  $W^u_i = L_i - L_{i-1}$.
\end{proof}

\begin{teo}
  \label{prost}
  A measure $\nu\in \proba(\PP^n)$ is stable (respectively
  semi-stable) with respect to the $\Sl(n+1, \C)$-action if and only
  if for any proper linear subspace $L \subset \PP^n$
  \begin{gather}
    \label{eq:prost}
    \nu(L) < \frac{\dim L +1}{n+1} \quad \Big ( \text{respectively }\,
    \nu(L) \leq \frac{\dim L +1}{n+1} \Big ).
  \end{gather}
\end{teo}

\begin{proof}
  Fix $\nu \in \proba(\PP^n)$. Assume first that the strict inequality
  holds in \eqref{eq:prost}. Fix $v\in S(\su(n+1) )$ and use the
  notation fixed above. Set $a_i : = \nu(L_i)$, so
  $0=a_{-1} \leq a_0 \leq \cds \leq a_r =1$ and
  $\nu(W^u_i) = a_i - a_{i-1}$ by the previous lemma.  By Theorem
  \ref{stabile} to prove that $\nu$ is stable it is enough to show
  that $\la_\nu >0$ on $\tits$, where $X=\Sl(n+1, \C) / \SU(n+1)$.  We
  apply Theorem \ref{calcolo-lambda} to compute $\la_\nu(\e(-v))$:
  \begin{gather*}
    \la_\nu (\e(-v) ) = \sum_{i=0}^r c_i (a_i - a_{i-1}) =
    \sum_{i=0}^r c_i a_i - \sum_{i=1}^rc_i a_{i-1}=\\
    = c_r + \sum_{i=0}^{r-1} a_i (c_i - c_{i+1}) .
  \end{gather*}
  Set $\eps_i:= \dim E_i$. By \eqref{eq:prost} $a_i < \eps_i / (n+1)$.
  Since $c_i - c_{i+1} < 0$ we get
  \begin{gather*}
    \la_\nu (\e(-v)) > c_r + \frac{1}{n+1}
    \sum_{i=0}^{r-1} (c_i - c_{i+1}) \eps_i .  \\
    \sum_{i=0}^{r-1} (c_i - c_{i+1}) \eps_i = c_0 \eps_0 +
    \sum_{i=1}^{r-1} c_i (\eps_i - \eps_{i-1}) -c_r \eps_{r-1} =
    \\
    = \sum_{i=0}^{r-1} c_i \dim V_i - c_r \eps_{r-1}.
  \end{gather*}
  Since $v\in \su(n+1)$,
  \begin{gather*}
    \sum_{i=0}^ r c_i \dim V_i =\operatorname{tr} v=0.
  \end{gather*}
  So
  \begin{gather*}
    \sum_{i=0}^{r-1} c_i \dim V_i - c_r \eps_{r-1} = -c_r \dim V_r -
    c_r \eps_{r-1} = - (n+1) c_r.
  \end{gather*}
  Summing up
  \begin{gather*}
    \la_\nu (\e(-v) ) > c_r - \frac{ (n+1) c_r}{n+1} =0.
  \end{gather*}
  So \eqref{eq:prost} implies that $\la_\nu >0$ on $\tits$, hence that
  $\nu$ is stable. The same computation using Theorem \ref
  {semi-stable-abstract} yields the result for semi-stability.
  Conversely let us prove that condition \eqref{prost} is necessary
  for stability. Let $L \subset \PP^n$ be a proper linear subspace of
  dimension $d$. Let $V_0 \subset \C^{n+1}$ be such that $L=\PP(V_0)$
  and let $V_1$ be the orthogonal complement to $V_0$. Thus
  $\dim V_0 = d+1$, $\dim V_1= n-d$.  Set $c_0 = (d-n) $ and
  $c_1 = d+1$. Let $v$ be the operator that acts as multiplication by
  $-ic_j$ on $V_j$.  Since $d< n$, $v \in \su(n+1)$.  If
  $\nu \in \proba(\PP^n) $ is stable, then
  \begin{gather*}
    0 < \la_\nu(\e(-v)) = c_0 \nu (L)) + c_1 (1 - \nu(L)) = c_1 -
    (c_1 - c_0) \nu (L),\\
    \text{hence } \quad \nu(L) < \frac{c_1}{c_1 -c_0} = \frac{d+1}{
      n+1}.
  \end{gather*}
  Thus \eqref{eq:prost} is a necessary condition for stability.  Also
  in this case the argument for semi-stability is identical.
\end{proof}
The next result gives a complete characterization of the measures on
$\PP^n$ that are polystable with respect to the standard action of
$\Sl(n+1,\C)$.
\begin{teo}
  \label{props}
  Let $\nu \in \proba(\PP^n)$. Then $\nu$ is polystable if and only if
  there exists a splitting $\C^{n+1}=V_0\oplus \cdots \oplus V_r$ and
  measures $\nu_j \in \proba (\PP (V_j))$ such that $\nu_j$ is stable
  with respect to $\mathrm{SL}(V_j,\C))$ for $j=0,\ldots,r$ and
  \begin{gather*}
    \nu=\sum_{j=0}^r \frac{\dim V_j }{n+1} \nu_j.
  \end{gather*}
\end{teo}
We start with two technical lemmata.
\begin{lemma}\label{tecnico}
  Let $\nu\in \proba(\PP^n )$ be such that $\fun (\nu)=0$. Then there
  exists an orthogonal splitting
  $\C^{n+1}=V_0\oplus \cdots \oplus V_r$ such that $\nu$ is
  concentrated on $\PP (V_0)\cup \cdots \cup \PP (V_r)$ and it is
  stable with respect to
  $\mathrm{SL}(V_0 , \C) \times \cdots \times \mathrm{SL}(V_r , \C)$.
\end{lemma}
\begin{proof}
  Since $\nu$ is polystable $\la_\nu \geq 0$.  If $Z(\nu)=\emptyset$,
  then $\nu$ is stable and the theorem is trivially proved.  Assume
  $Z(\nu)\neq \vacuo$. Let $v\in Z(\nu)$ be such that
  \begin{gather*}
    \dim T_v =\mathrm{max}_{w\in \e^{-1}(Z(\nu))} \dim T_{w},
  \end{gather*}
  where $T_w=\overline{\exp(\R w)}$. By the same argument as in the
  proof of Lemma \ref{lemmozzo}, $\exp(\C v )\subset G_{\nu}$.  If
  $K'$ is a compact subgroup of $K^v$ such that $T_v \cdot K'=K^v$ and
  $K'\cap T_v$ is finite, then $\nu$ is stable with respect to
  $G'=(K')^{\C}$. Let $ \limes :M \lra \mathrm{Crit}(\mu^v) $ be the
  map defined in \eqref{def-limes}.  By Proposition \ref{limite} it is
  well-defined.  If $f \in C(M)$, then
  $\lim_{t\to +\infty} f (\exp(itv)\cd x ) = f (\alfa(x))$.  It
  follows that $\exp(itv) \cd \nu \weak \alfa \cd \nu $ for
  $t\to + \infty$.  Since $\exp(\C v)\subset G_\nu$, we conclude that
  $\nu= \limes\cd \nu$.  So $\nu$ is concentrated on the image of
  $\alfa$, i.e. on the critical sets of $\mu^v$.  Using the notation
  of \ref{autospazi} and Lemma \ref{proiettivo} this means that $\nu$
  is concentrated on $\PP (V_0)\cup \cdots \cup \PP (V_r)$.  Moreover,
  $\mathrm{SU} (V_0 )\times \cdots \times \mathrm{SU} (V_j)\subset
  K^v$
  and the intersection
  $(\mathrm{SU}(V_0 ) \times \cdots \times \mathrm{SU}(V_r )) \cap
  T_v$
  is finite due the fact that $T_v$ is contained in the center of
  $K^v$.  This proves that we may choose $K'$ so that
  $\mathrm{SU}(V_0) \times \cdots \times \mathrm{SU}(V_r )\subset K'$.
  Since $\nu$ is stable for $G'$ then it is stable for
  $\mathrm{SL}(V_0 , \C) \times \cdots \times \mathrm{SL}(V_r ,
  \C)\subset G'$ concluding the proof.
\end{proof}

\begin{say} \label{say-metriche-Pn}
  \label{herm} Let $V$ be a complex vector space and let $h $ be a
  Hermitian product on $V$. The group $\SU (V, h) $ acts on $\PP(V)$
  with momentum mapping $\mu$ given by the formula
  \begin{gather*}
    \mu([v]) = -{i} \left ( P_v - \frac{\id_V} {n+1}\right),
  \end{gather*}
  where $P_v$ denotes the $h$-orthogonal projection $V\ra \C v$.  The
  construction in Section \ref{sez-misure} yields a momentum mapping
  $\fun: \proba(\PP(V)) \ra \su(V, h)$.  If $h'$ is a second Hermitian
  product on $V$, and $h' (z, z') = h (Lz, Lz') $ with $L \in \Gl(V) $
  $h$-self-adjoint, then $L \circ P'_v = P_{Lv} \circ L$, so
  $ \Ad L \circ \mu' = \mu \circ L$. Similarly
  $ \Ad L \circ \fun' = \fun \circ L$, so
  $\fun\meno(0) = L \left ( (\fun')\meno(0) \right )$. It follows
  immediately that the stability of a measure $\nu \in \proba(\PP(V))$
  does not depend on the choice of $h$.  This problem is central in
  \cite{teleman-symplectic-stability}.  It would be interesting to
  develop the arguments in that paper in the setting of Section
  \ref{sec:abstract-setting}.
\end{say}

\begin{lemma}\label{misure-concentrate-polistabile}
  Let $\nu\in \proba(\PP^n)$. Let
  $\C^{n+1}=V_0 \oplus \cdots \oplus V_r$ be an orthogonal splitting
  such that $\nu$ is concentrated on
  $\PP^n (V_0) \cup \cdots \cup \PP (V_r)$. Then $\fun (\nu)=0$ if and
  only if $\nu(\PP(V_j))=\dim V_j/(n+1)$ and $\fun_j (\nu_j)=0$, where
  $\fun_j : \proba(\PP(V_j)) \ra \su (V_j)$ is the momentum mapping
  with respect to the natural $\SU(V_j , \C )$-action on $\PP(V_j)$.
\end{lemma}
\begin{proof}
  By assumption $\nu=\sum_{j=0}^r a_j \nu_j$, where $\nu_j$ is a
  probability measure on $\PP (V_j)$, $a_j:=\nu(\PP(V_j))\geq 0$ and
  $\sum_{j=0}a_j=1$.  Let $\mu$ be the momentum mapping on $\PP^n$ and
  $\mu_j $ the momentum mapping on $\PP(V_j)$ for the natural
  $\SU(V_j)$-action.  Set for simplicity $V:=\C^{n+1}$ and
  $d_j:= \dim V_j$.  If $[z] \in \PP(V_j)$, then
  \begin{gather*}
    \mu([z]) = \mu_j([z]) + {i} \left ( \frac{\id_V}{n+1} -
      \frac{\id_{V_j}} { d_j} \right ).
  \end{gather*}
  Hence
  \begin{gather*}
    \fun (\nu) = \sum_{j=0}^r a_j \int_{\PP(V_j)} \mu([z]) d\nu_j ([z]) = \\
    = \sum_{j=0}^r a_j \fun_j (\nu_j) + {i}
    \sum_{j=0}^r a_j \left ( \frac{\id_V}{n+1} - \frac{\id_{V_j}} { d_j} \right ) =\\
    = \sum_{j=0}^r a_j \fun_j (\nu_j) + {i} \left ( \frac{\id_V}{n+1}
      - \sum_{j=0}^r \frac{a_j } { d_j} \id_{V_j} \right ) =
    \\
    = \sum_{j=0}^r a_j \fun_j (\nu_j) + {i} \sum_{j=0}^r \left (
      \frac{1}{n+1} - \frac{a_j } { d_j} \right ) \id_{V_j} .
  \end{gather*}
  Since $\fun_j(\nu_j) \in \su (V_j)$ the terms in the last sum are
  all orthogonal to each other.  Thus $\fun(\nu) = 0$ if and only if
  every term vanishes.
\end{proof}

\begin{proof}[Proof of Theorem \ref{props}]
  If $\nu$ is stable, then the theorem holds with $r=0$.  Assume that
  $\nu$ is polystable but not stable. Then there exists
  $g\in \mathrm{SL}(n+1,\C)$ such that $\fun (g\cdot \nu)=0$. Set
  $\nu'=g\cdot \nu$. By Lemma \ref{tecnico} there exists an orthogonal
  splitting $\C^{n+1}=W_0 \oplus \cdots \oplus W_r$ such that $\nu'$
  is concentred on $\PP (W_0)\cup \cdots \cup \PP (W_r)$ and $\nu'$ is
  stable with respect to
  $G:=\mathrm{SL}(W_0 , \C) \times \cdots \times \mathrm{SL}(W_r ,
  \C)$.
  Therefore $\nu'=\sum_{j=0}^r a_j \nu_j'$ where
  $\nu_j' \in \proba (\PP (W_j))$ and $\sum_{j=0}^r a_j=1$.  By Lemma
  \ref{misure-concentrate-polistabile}, $a_j={d_j}/(n+1)$ and
  $\fun_j (\nu_j')=0$. In particular $\nu'_j$ is polystable with
  respect to $\mathrm{SL}(W_j , \C)$ for $j=0,\ldots,r$.  The
  stabilizer of $\nu'_j$ in $\SL(W_j)$ is contained $G_\nu$. Since
  $\nu$ is $G$-stable, $G_\nu$ is compact, so the same holds for the
  stabilizer in $\SL(W_j)$ of $\nu'_j$. Hence $\nu'_j$ is actually
  stable with respect to $\SL(W_j)$ for any $j$.  Set
  $V_j:=g^{-1} (W_j)$ for $j=0,\ldots,r$.  Then
  $\C^{n+1} = V_0\oplus \cds \oplus V_r$.  The measures
  $\nu_j := g\meno \cd \nu'_j \in \proba (\PP(V_j))$ are stable with
  respect to $\SL(V_j)$ since $g$ is an isomorphism and as observed in
  \ref{say-metriche-Pn} we do not need to care about the Hermitian
  product. Finally
  \begin{gather}
    \label{umpfi-2}
    \nu = \sum_{j=0}^r \frac{d_j}{n+1} \nu_j.
  \end{gather}
  We have proved that the condition in the theorem is necessary for
  polystability.  Vice versa assume that there exists a splitting
  $\C^{n+1}=V_0 \oplus \cdots \oplus V_r$ such that \eqref{umpfi-2}
  holds, where $\nu_j\in \proba (\PP (V_j ))$ is stable. Fix a
  Hermitian scalar product $h$ on $V=\C^{n+1}$, such that the above
  splitting is orthogonal.  By Lemma
  \ref{misure-concentrate-polistabile} $\nu$ is polystable when we
  consider on $\proba(\PP^n)$ the action of $\SU(V, h)$. As noted in
  \ref{herm} the choice of the Hermitian product does not matter. So
  $\nu$ is polystable also with respect to the standard action of
  $\SU(n+1) $.
\end{proof}

\begin{say}\label{recover}
  In \cite[\S 2.2] {donaldson-numerical-results} Donaldson considers
  the $\nu$-balanced metrics on $\OO_{\PP^n}(1)$ where
  $\nu \in \proba (\PP^n)$. He establishes two conditions on $\nu$
  ensuring that there exists a $\nu$-balanced metric. The existence of
  a $\nu$-balanced metric is equivalent to the polystability of $\nu$
  with respect to the action of $\Sl (n+1, \C)$ on $\PP^n$.  We now
  show how to recover Donaldson's conditions. The two conditions are
  the following:
  \begin{enumerate}
  \item for any non-trivial linear function $\la $ on $ \C^{n+1}$ the
    function $\log ( \frac{|\la(z)|} {|z|})$ is $\nu$-integrable.
  \item $\nu = \sum_{i=1}^r \la_i \delta_{y_i}$ with $\la_i \geq 0$
    and $\sum_{i=1}^r \la_i = 1$ and such that
    \begin{gather*}
      \frac{\nu( L)}{\dim L + 1 } < \frac{1}{n+1},
    \end{gather*}
  \end{enumerate}
  for any proper linear subspace $L\subset \PP^n$.  We prove that if
  $\nu \in \proba (\PP^n)$ satisfies one of the above conditions, then
  it is stable.  If the first condition holds, then $\nu(H)=0$ for any
  hyperplane. Hence it satisfies the assumption of Theorem \ref{prost}
  and so it is stable. The second condition is also clearly covered by
  Theorem \ref{prost}.
\end{say}
\begin{say}
  \label{zombro-1}
  Millson and Zombro \cite{millson-zombro} studied measures on the
  sphere $S^2\subset \R^3$. If $i:S^2 \hookrightarrow \R^3$ denotes
  the inclusion, they studied the following map
  \begin{gather*}
    \proba(S^2) \lra \R^3 \qquad \nu \mapsto \int_{S^2} i(x) d\nu(x),
  \end{gather*}
  that assigns to a measure its center of mass.  Since the inclusion
  $i$ is the momentum mapping $\mu$ for the $\mathrm{SO}(3)$-action on
  $S^2$, this is exactly the map \eqref{def-momento-misure}.  Their
  main result asserts that if $\nu$ has no atom of mass greater than
  or equal to ${1}/{2}$, then there exists $g\in \mathrm{PGL}_2 (\C)$
  such that $\fun (g\nu)=0$. This follows directly from Theorem
  \ref{prost}. In fact Theorem \ref{prost} says that this condition is
  equivalent to stability. The proof in \cite{millson-zombro} is based
  on the notion of conformal center of mass. This technique is rather
  complicated and seems to be limited to $S^2$.  We also point out
  that Corollary \ref{stable-dense} (relying on Corollary \ref
  {stable-open-abstract-setting}) generalizes Proposition 3.3 of
  \cite{millson-zombro}. Furthermore Millson and Zombro defined a
  measure $\nu \in \proba(S^2)$ to be \emph{semi-stable} if there is
  no atom of mass greater than ${1}/{2}$, and \emph{nice semi-stable}
  if it has two atoms each of mass ${1}/{2}$. By Theorem \ref{prost}
  our notion of semi-stable coincides with theirs and by Theorem
  \ref{props} a measure is nice semi-stable if and only if it is
  polystable but not stable.
\end{say}
\begin{say}\label{flag-stability}
  Let $M\subset \mathbb P^n$ be a projective manifold endowed by the
  Fubini study metric. Let $K=\{g\in \mathrm{SU}(n+1):\, g(M)=M\}$ and
  let $G=K^\C$.  The $K$-action on $M$ is Hamiltonian and a momentum
  mapping is given by the restriction to $M$ of the momentum mapping
  of $\PP^n$.  Let $\nu \in \proba{(M)}$ be such that for any linear
  subspace $L\subset \mathbb P^n$,
  \begin{gather}\label{eq:flag}
    \nu(M\cap L) < \frac{\dim L +1}{n+1} \left(\mathrm{respectively}\,
      \nu(M\cap L)\leq \frac{\dim L +1}{n+1}\right).
  \end{gather}
  If $\xi \in \liek$, then the unstable manifolds are given by
  $W_j^\xi =(M\cap L_j )-( M \cap L_{j-1})$. Hence the computation of
  $\la_\nu(\e(-v))$ works as for the projective space $\mathbb P^n$,
  showing that $\nu$ is stable, respectively semi-stable.
\end{say}
\begin{say}\label{eigen}
  We wish to recall briefly the original motivation for the map
  $F_\nu$ defined in \eqref {def-bly}.  Let $(M,g)$ be a compact
  Riemannian manifold. Given functions $f_1, \lds, f_r \in \cinf (M)$
  such that $\int_M f_j \vol_g =0$, the Rayleigh theorem yields the
  upper bound
  \begin{gather}
    \label{rayleigh}
    \la_1 (M,g) \leq \frac{\sum_{j=1}^r \int_N |\nabla f_j|^2_g \vol_g
    } {\sum_{j=1}^r \int_N f_j^2 \vol_g }.
  \end{gather}
  Assume that $M$ admits a special metric $\go$ and that
  $\fo_j , \lds, \fo_r$ are eigenfunctions of $\la_1$ for this metric.
  Assume moreover that there is a large group $G$ acting on $M$.  If
  $g$ is another Riemannian metric on $M$, one might look for
  functions $f_j$ of the form $f_j= a^* \fo_j$ for some $a\in G$.  If
  there is some $a\in G$ that makes the integral of all these
  functions vanish, one gets the bound \eqref{rayleigh}.  If moreover
  one is able to compute the right hand side of the bound, then one
  gets an interesting estimate.

  In the paper \cite{hersch} of Hersch $M=S^2$, $\go$ is the round
  metric and $G=\operatorname{PSL}(2, \C)$ acting by M\"obius
  transformations.  The functions $\fo_j$ are the three coordinate
  functions $x,y,z$. Hersch was able to show that if $g$ is an
  arbitrary Riemannian metric on $S^2$ (normalized to have volume
  $4\pi$), then there is $a\in G$ such that
  $\int_M a^*x \vol_g = \int_M a^*y \vol_g =\int_M a^*z \vol_g=0$.
  Moreover he was able to show that the right hand side in
  \eqref{rayleigh} is equal to 2. Thus he showed that
  $\la_1(S^2, g) \leq 2$.

  Bourguignon, Li and Yau \cite{bourguignon-li-yau} realized that this
  method applies also to estimate $\la_1(\PP^n(\C), g)$ if $g$ is a
  \emph{K\"ahler} metric.  In \cite{biliotti-ghigi-American} we recast
  the method of Hersch-Bouguignon-Li-Yau in terms of momentum mapping
  and applied it when $M$ is an arbitrary compact Hermitian symmetric space,
  $\go$ is the symmetric metric, $G=\Aut(M)$ and the functions $\fo_j$
  are the components of the momentum mapping $\mu: M \ra \liek^*$ for
  $K:= \operatorname{Isom}(M, \go)$.  (Related papers include
  \cite{arezzo-ghigi-loi}, \cite{biliotti-ghigi-AIF}, \cite {legrosa}
  and \cite{panelli-podesta}).

  In \cite{hersch}, \cite{bourguignon-li-yau} and
  \cite{biliotti-ghigi-American} the estimation of $\la_1$ proceeds in
  two steps: the first one is to find $a\in G$ such that
  $\int_M a^*f_j\vol_g=0$.  The second step is to actually compute the
  right hand side in \eqref{rayleigh}.  The map \eqref{def-bly} is the
  tool to deal with the first step.  Set $\nu:=\vol_g / \Vol(M,g)$.
  Since $\fo_j$ are the components of $\mu$,
  $\int_M a^*\fo_j \vol_g = 0$ for all $j$ if and only if
  $\int_M \mu(ax) d\nu(x) =0$. Thus the first step amounts to proving
  that the measure $\nu$ is $G$--polystable!  Theorem \ref{W-stabile}
  represents a very general solution to the first step. The estimate
  we get is the following one.
\end{say}
\begin{teo}
  \label{stimaccia}
  Let $(M,g, \om)$ be a \Keler manifold and let $K$ be a compact group
  acting isometrically and almost effectively on $M$ with momentum
  mapping $\mu: M \ra \liek^*$.  Fix an $\Ad$-invariant scalar product
  on $\liek$ and let $e_1 , \lds, e_r$ be an orthonormal basis of
  $\liek$. Set $\mu_j := \sx \mu, e_j\xs$.  If $g$ is any K\"ahler
  metric on $M$ then there is $a\in G$ such that
  \begin{gather}
    \label{eq:stima}
    \la_1(M,g) \leq \frac{\sum_{j=1}^r \int_M |\nabla (a^*f_j)|^2
      \vol_g } {\int_M a^*(|\mu|^2) \vol_g}
  \end{gather}
\end{teo}

\begin{say}
  The second step is more mysterious.  At the moment we are not able
  to compute the right hand side in \eqref{eq:stima} in any reasonable
  geometric situation, except the known ones, i.e. Hermitian symmetric
  spaces.  We believe that this computation can be carried out in much
  greater generality and that it would yield very interesting
  estimates. We leave this problem for future investigations. It is
  important to notice that in the case of symmetric spaces the
  K\"ahler-Einstein metric (i.e. the symmetric metric) maximizes
  $\la_1$ among K\"ahler metrics in $\chern_1(M)$.  Now thanks to the
  work of Apostolov, Jakobson and Kokarev there are examples of Fano
  manifolds where this does not happen, see
  \cite[Cor. 3.4]{apostolov-jakobson-kokarev}.

\end{say}

\begin{say}
  \label{flag}

  We want to explain the relation between the results in
  \cite{biliotti-ghigi-American} and those in the present paper.  Let
  $M$ be a flag manifold. This means that $K$ acts transitively on the
  complex manifold $M$.  By Bott-Borel-Weyl theorem, any flag manifold
  is the unique complex $K$-orbit in $\mathbb P (V)$ for some
  irreducible representation $\tau: G \longrightarrow \mathrm{GL}(V)$.
  In \cite{biliotti-ghigi-American} we defined $\tau$-admissible
  measure as those $\nu\in \pb$ such that $\nu(H\cap M)=0$ for any
  hyperplane $H\subset \mathbb P(V)$.  It is immediate that a
  $\tau$--admissible measure satisfies condition \eqref{eq:flag}, so
  it is stable.  In \cite[Thm. 3]{biliotti-ghigi-American} we proved
  that for a $\tau$-admissible $\nu$ the map $F_\nu:G \lra \into$ is
  surjective.  By (\ref{flag-stability}) the complement of the open
  unstable manifold of any $v \in \liek$ is contained in a hyperplane
  section of $M$.  So if $\nu$ is $\tau$-admissible, it is
  concentrated on the open unstable manifold. Therefore
  $\nu \in \W (M, K, \omega, \mu)$.  This means that the assumptions
  in Theorem \ref{T-propria} (i.e. when restricting to a torus) are
  weaker than those of \cite[Thm. 3]{biliotti-ghigi-American}.
  Moreover the conclusion is stronger since we prove that $F_\nu$ is a
  surjective submersion.  In the nonabelian case, Theorem
  \ref{propria} treats a slightly smaller class of measures.  The
  proofs are completely different. We expect that a more general
  surjectivity result holds in the nonabelian case, but we leave it
  for future investigations.
\end{say}

\begin{say}
  In \cite{biliotti-ghigi-American} we used the map $F_{\nu_0}$, where
  $\nu_0$ is the $K$-invariant measure on a flag manifold, to get a
  diffeomorphism between $X$ and $\intec$. The following proposition
  shows that this holds in greater generality.
\end{say}

  \begin{prop}\label{compactification}
    Let $(M,\om)$ be a \Keler manifold and $K$ a compact connected Lie
    group. Assume that $K$ acts almost effectively on $M$ with
    momentum mapping $\mu$.  If $\nu_0\in\ac (M)$ is $K$--invariant,
    then $F_{\nu_0}$ descends to a map
    \begin{gather*}
      \tilde{F}_{\nu_0}: X=G/K \ra \liek
    \end{gather*}
    that is a diffeomorphism of the symmetric space $X$ onto $\into$.
  \end{prop}
  \begin{proof}
    Since $k \cd \nu_0 = \nu_0$,
    \begin{gather*}
      F_{\nu_0} (ak) = \int_M \Mom (ak\cd x) d\nu_0(x) = \int _M
      \Mom(a\cd y ) d\nu_0(y) = F_{\nu_0} (a).
    \end{gather*}
    Hence $F_{\nu_0}$ descends to a map on $G/K$. By Theorem
    \ref{propria} it is a local diffeomorphism and a proper map. Hence
    a covering map. Since $\into$ is contractible, it follows that
    $F_{\nu_0}$ is a diffemorphism onto $\into$.
  \end{proof}

\begin{say}
  In \cite[Thm. 2, p. 239]{biliotti-ghigi-American} we proved that if
  $M=K/K_o$ is the flag manifold given by the complex $K$-orbit in
  $\PP(V)$ for an irreducible representation $\tau : G \ra \Gl(V)$,
  then the map $F_{\nu_0}$ extends to the Satake compactification
  $ \XS$ of $ X = G/K$ yielding a homeomorphism between $\XS$ and the
  convex hull of the momentum image of $M$, which is a coadjoint
  orbit. Such a homeomorphism exists for any symmetric space of
  non-compact type by a theorem of Kor\'anyi \cite{koranyi-remarks}.
  If $M$ is an arbitrary compact K\"ahler manifold with a Hamiltonian
  action of $K$, by Proposition \ref{compactification} the convex hull
  $\convo$ is a $K$-equivariant compactification of $G/K$ where
  $G=K^\C$.  In \cite{bgh-israel-p} we completely described the faces
  of $\convo$. More precisely we proved that the faces of $\convo$ are
  exposed and correspond to maxima of components of the momentum
  mapping. We used this fact to realize a close connection between the
  faces of $\convo$ and parabolic subgroups of $G$. Hence, as for the
  Satake compatifications, boundary components of $G/K$ are related to
  parabolic subgroups of $G$. We think that it would be interesting to
  further analyze these connections.
\end{say}

\def\cprime{$'$}


\begin{thebibliography}{10}

\bibitem{apostolov-jakobson-kokarev} V.~Apostolov, D.~Jakobson, and
  G.~Kokarev.  \newblock An extremal eigenvalue problem in {K}\"ahler
  geometry.  \newblock {\em J. Geom. Phys.}, 91:108--116, 2015.

\bibitem{arezzo-ghigi-loi} C.~Arezzo, A.~Ghigi, and A.~Loi.  \newblock
  Stable bundles and the first eigenvalue of the {L}aplacian.
  \newblock {\em J. Geom. Anal.}, 17(3):375--386, 2007.

\bibitem{atiyah-commuting} M.~F. Atiyah.  \newblock Convexity and
  commuting {H}amiltonians.  \newblock {\em Bull. London Math. Soc.},
  14(1):1--15, 1982.

\bibitem{audin-torus-actions} M.~Audin.  \newblock {\em Torus actions
    on symplectic manifolds}, volume~93 of {\em Progress in
    Mathematics}.  \newblock Birkh\"auser Verlag, Basel, revised
  edition, 2004.

\bibitem{azad-loeb-bulletin} H.~Azad and J.-J. Loeb.  \newblock
  Plurisubharmonic functions and the {K}empf-{N}ess theorem.
  \newblock {\em Bull. London Math. Soc.}, 25(2):162--168, 1993.

\bibitem{berger-geometry-1} M.~Berger.  \newblock {\em Geometry. {I}}.
  \newblock Universitext. Springer-Verlag, Berlin, 1987.  \newblock
  Translated from the French by M. Cole and S. Levy.

\bibitem{biliotti-ghigi-AIF} L.~Biliotti and A.~Ghigi.  \newblock
  Homogeneous bundles and the first eigenvalue of symmetric spaces.
  \newblock {\em Ann. Inst. Fourier (Grenoble)}, 58(7):2315--2331,
  2008.

\bibitem{biliotti-ghigi-American} L.~Biliotti and A.~Ghigi.  \newblock
  Satake-{F}urstenberg compactifications, the moment map and
  {$\lambda_1$}.  \newblock {\em Amer. J. Math.}, 135(1):237--274,
  2013.

\bibitem{bgh-israel-p} L.~Biliotti, A.~Ghigi, and P.~Heinzner.
  \newblock Invariant convex sets in polar representations.  \newblock
  {\em Preprint, arXiv:1411.6041}, 2014.

\bibitem{borel-ji-libro} A.~Borel and L.~Ji.  \newblock {\em
    Compactifications of symmetric and locally symmetric spaces}.
  \newblock Mathematics: Theory \& Applications. Birkh\"auser Boston
  Inc., Boston, MA, 2006.

\bibitem{bourguignon-li-yau} J.-P. Bourguignon, P.~Li, and S.-T. Yau.
  \newblock Upper bound for the first eigenvalue of algebraic
  submanifolds.  \newblock {\em Comment. Math. Helv.}, 69(2):199--207,
  1994.

\bibitem{donaldson-numerical-results} S.~K. Donaldson.  \newblock Some
  numerical results in complex differential geometry.  \newblock {\em
    Pure Appl. Math. Q.}, 5(2, Special Issue: In honor of Friedrich
  Hirzebruch. Part 1):571--618, 2009.

\bibitem{dugundji} J.~Dugundji.  \newblock {\em Topology}.  \newblock
  Allyn and Bacon Inc., Boston, Mass., 1966.

\bibitem{duistermaat-kolk-Lie} J.~J. Duistermaat and J.~A.~C. Kolk.
  \newblock {\em Lie groups}.  \newblock
  Universitext. Springer-Verlag, Berlin, 2000.



\bibitem{dunford-schwartz-1} N.~Dunford and J.~T. Schwartz.  \newblock
  {\em Linear {O}perators. {I}. {G}eneral {T}heory}.  \newblock With
  the assistance of W. G. Bade and R. G. Bartle. Pure and Applied
  Mathematics, Vol. 7. Interscience Publishers, Inc., New York;
  Interscience Publishers, Ltd., London, 1958.

\bibitem{folland-real-analysis} G.~B. Folland.  \newblock {\em Real
    analysis}.  \newblock Pure and Applied Mathematics. John Wiley \&
  Sons, New York, second edition, 1999.

\bibitem{gangbo-kim-pacini} W.~Gangbo, H.~K. Kim, and T.~Pacini.
  \newblock Differential forms on {W}asserstein space and
  infinite-dimensional {H}amiltonian systems.  \newblock {\em
    Mem. Amer. Math. Soc.}, 211(993):vi+77, 2011.

\bibitem{georgula} V.~Georgulas, J.~W. Robbin, and D.~A. Salamon.
  \newblock The moment-weight inequality and the {H}ilbert-{M}umford
  criterion.  \newblock {\em Preprint. arXiv:1311.0410}, 2013.

\bibitem{guglielmino-torico} V.~Guillemin.  \newblock {\em Moment maps
    and combinatorial invariants of {H}amiltonian {$T\sp n$}-spaces},
  volume 122 of {\em Progress in Mathematics}.  \newblock Birkh\"auser
  Boston Inc., Boston, MA, 1994.


\bibitem{guillemin-sternberg-convexity-1} V.~Guillemin and
  S.~Sternberg.  \newblock Convexity properties of the moment mapping.
  \newblock {\em Invent. Math.}, 67(3):491--513, 1982.


\bibitem{heinzner-GIT-stein} P.~Heinzner.  \newblock Geometric
  invariant theory on {S}tein spaces.  \newblock {\em Math. Ann.},
  289(4):631--662, 1991.

\bibitem{heinzner-huckleberry-Inventiones} P.~Heinzner and
  A.~Huckleberry.  \newblock K\"ahlerian potentials and convexity
  properties of the moment map.  \newblock {\em Invent. Math.},
  126(1):65--84, 1996.

\bibitem{heinzner-huckleberry-MSRI} P.~Heinzner and A.~Huckleberry.
  \newblock Analytic {H}ilbert quotients.  \newblock In {\em Several
    complex variables (Berkeley, CA, 1995--1996)}, volume~37 of {\em
    Math. Sci. Res. Inst. Publ.}, pages 309--349. Cambridge
  Univ. Press, Cambridge, 1999.

\bibitem{heinzner-huckleberry-loose} P.~Heinzner, A.~T. Huckleberry,
  and F.~Loose.  \newblock K\"ahlerian extensions of the symplectic
  reduction.  \newblock {\em J. Reine Angew. Math.}, 455:123--140,
  1994.

\bibitem{heinzner-loose} P.~Heinzner and F.~Loose.  \newblock
  Reduction of complex {H}amiltonian {$G$}-spaces.  \newblock {\em
    Geom. Funct. Anal.}, 4(3):288--297, 1994.

\bibitem{heinzner-schwarz-Cartan} P.~Heinzner and G.~W. Schwarz.
  \newblock Cartan decomposition of the moment map.  \newblock {\em
    Math. Ann.}, 337(1):197--232, 2007.

\bibitem{heinzner-schwarz-stoetzel} P.~Heinzner, G.~W. Schwarz, and
  H.~St{\"o}tzel.  \newblock Stratifications with respect to actions
  of real reductive groups.  \newblock {\em Compos. Math.},
  144(1):163--185, 2008.

\bibitem{heinzner-stoetzel-global} P.~Heinzner and H.~St{\"o}tzel.
  \newblock Critical points of the square of the momentum map.
  \newblock In {\em Global aspects of complex geometry}, pages
  211--226.  Springer, Berlin, 2006.

\bibitem{hersch} J.~Hersch.  \newblock Quatre propri\'et\'es
  isop\'erim\'etriques de membranes sph\'eriques homog\`enes.
  \newblock {\em C. R. Acad. Sci. Paris S\'er. A-B}, 270:A1645--A1648,
  1970.

\bibitem{hirsch-differential-topology} M.~W. Hirsch.  \newblock {\em
    Differential topology}.  \newblock Springer-Verlag, New York,
  1976.  \newblock Graduate Texts in Mathematics, No. 33.

\bibitem{kapovich-leeb-millson-convex-JDG} M.~Kapovich, B.~Leeb, and
  J.~Millson.  \newblock Convex functions on symmetric spaces, side
  lengths of polygons and the stability inequalities for weighted
  configurations at infinity.  \newblock {\em J. Differential Geom.},
  81(2):297--354, 2009.

\bibitem{kempf-ness} G.~Kempf and L.~Ness.  \newblock The length of
  vectors in representation spaces.  \newblock In {\em Algebraic
    geometry (Proc. Summer Meeting, Univ. Copenhagen, Copenhagen,
    1978)}, volume 732 of {\em Lecture Notes in Math.}, pages
  233--243. Springer, Berlin, 1979.



\bibitem{koranyi-remarks} A.~Kor{\'a}nyi.  \newblock Remarks on the
  {S}atake compactifications.  \newblock {\em Pure Appl. Math. Q.},
  1(4, part 3):851--866, 2005.

\bibitem{legrosa} E.~Legendre and R.~Sena-Dias.  \newblock Toric
  aspects of the first eigenvalues.  \newblock {\em
    \texttt{arXiv:1505.01678}}, 2015.  \newblock Preprint.

\bibitem{mcduff-salamon-symplectic} D.~McDuff and D.~Salamon.
  \newblock {\em Introduction to symplectic topology}.  \newblock
  Oxford Mathematical Monographs. The Clarendon Press Oxford
  University Press, New York, second edition, 1998.

\bibitem{millson-zombro} J.~J. Millson and B.~Zombro.  \newblock A
  {K}\"ahler structure on the moduli space of isometric maps of a
  circle into {E}uclidean space.  \newblock {\em Invent. Math.},
  123(1):35--59, 1996.

\bibitem{mumford-GIT} D.~Mumford, J.~Fogarty, and F.~Kirwan.
  \newblock {\em Geometric invariant theory}, volume~34 of {\em
    Ergebnisse der Mathematik und ihrer Grenzgebiete (2)}.  \newblock
  Springer-Verlag, Berlin, third edition, 1994.

\bibitem{mundet-Crelles} I.~Mundet~i Riera.  \newblock A
  {H}itchin-{K}obayashi correspondence for {K}\"ahler fibrations.
  \newblock {\em J. Reine Angew. Math.}, 528:41--80, 2000.

\bibitem{mundet-Trans} I.~Mundet~i Riera.  \newblock A
  {H}ilbert-{M}umford criterion for polystability in {K}aehler
  geometry.  \newblock {\em Trans. Amer. Math. Soc.},
  362(10):5169--5187, 2010.

\bibitem{panelli-podesta} F.~Panelli and F.~Podest{\`a}.  \newblock On
  the first eigenvalue of invariant {K}\"ahler metrics.  \newblock
  {\em Math. Z.}, 281(1-2):471--482, 2015.

\bibitem{schneider-convex-bodies} R.~Schneider.  \newblock {\em Convex
    bodies: the {B}runn-{M}inkowski theory}, volume~44 of {\em
    Encyclopedia of Mathematics and its Applications}.  \newblock
  Cambridge University Press, Cambridge, 1993.

\bibitem{sjamaar-Annals} R.~Sjamaar.  \newblock Holomorphic slices,
  symplectic reduction and multiplicities of representations.
  \newblock {\em Ann. of Math. (2)}, 141(1):87--129, 1995.

\bibitem{teleman-symplectic-stability} A.~Teleman.  \newblock
  Symplectic stability, analytic stability in non-algebraic complex
  geometry.  \newblock {\em Internat. J. Math.}, 15(2):183--209, 2004.

\bibitem{thomas-GIT} R.~P. Thomas.  \newblock Notes on {GIT} and
  symplectic reduction for bundles and varieties.  \newblock In {\em
    Surveys in differential geometry. {V}ol. {X}}, volume~10 of {\em
    Surv. Differ. Geom.}, pages 221--273. Int. Press, Somerville, MA,
  2006.

\bibitem{tian-libro} G.~Tian.  \newblock {\em Canonical metrics in
    {K}\"ahler geometry}.  \newblock Birkh\"auser Verlag, Basel, 2000.
  \newblock Notes taken by Meike Akveld.



\end{thebibliography}
\end{document}